\documentclass[arma,referee]{svjour}
\usepackage{graphics}
\usepackage{amsmath}
\usepackage{fancyhdr}
\usepackage{amssymb}
\usepackage{latexsym}
\usepackage{graphicx}
\usepackage[dvipdf]{epsfig}
\usepackage{natbib}
\usepackage{longtable}
\usepackage{makeidx}
\usepackage[hypertex]{hyperref}
\hypersetup{backref, colorlinks=true}
\usepackage{setspace}  
\def\rz{\ifmmode{I\hskip -3pt R}
    \else{\hbox{$I\hskip -3pt R$}}\fi}
\def\nz{\ifmmode{I\hskip -3pt N}
    \else{\hbox{$I\hskip -3pt N$}}\fi}
\def\gz{\ifmmode{Z\hskip -4.8pt Z}
    \else{\hbox{$Z\hskip -4.8pt Z$}}\fi}
\def\cz{\ifmmode{C\hskip -4.8pt\vrule height5.8pt\hskip6.3pt}
\else{\hbox{$C\hskip -4.8pt\vrule height6.0pt\hskip6.3pt$}}\fi}
\def\qz{\ifmmode{Q\hskip -5.0pt\vrule height6.0pt depth0pt
   \hskip6pt}
    \else{\hbox{$Q\hskip -5.0pt\vrule height6.0pt depth0pt
   \hskip6pt$}}\fi}

\newcommand{\beq}{\begin{equation}}
\newcommand{\eeq}{\end{equation}}
\newcommand{\beql}{\begin{equation} \label}
\newcommand{\beqs}{\begin{eqnarray}}
\newcommand{\eeqs}{\end{eqnarray}}
\newcommand{\beas}{\begin{eqnarray*}}
\newcommand{\eeas}{\end{eqnarray*}}
\newcommand{\ber}{\begin{array}}
\newcommand{\eer}{\end{array}}
\newcommand{\becs}{\begin{cases}}
\newcommand{\eecs}{\end{cases}}

\newcommand{\leftm}{\left[\begin{array}}
\newcommand{\rightm}{\end{array}\right]}

\newcommand{\bfd}{{\bf d}}
\newcommand{\bfe}{{\bf e}}

\newcommand{\bfk}{{\bf k}}

\newcommand{\bfm}{{\bf m}}
\newcommand{\bfn}{{\bf n}}

\newcommand{\bfr}{{\bf r}}

\newcommand{\bfv}{{\bf v}}
\newcommand{\bfw}{{\bf w}}
\newcommand{\bfx}{{\bf x}}

\newcommand{\bfA}{{\bf A}}

\newcommand{\bfF}{{\bf F}}
\newcommand{\bfG}{{\bf G}}

\newcommand{\bfI}{{\bf I}}

\newcommand{\bfL}{{\bf L}}
\newcommand{\bfM}{{\bf M}}

\newcommand{\bfP}{{\bf P}}
\newcommand{\bfQ}{{\bf Q}}
\newcommand{\bfR}{{\bf R}}

\newcommand{\bfY}{{\bf Y}}

\newcommand{\calH}{{\cal H}}

\newcommand{\calK}{{\cal K}}
\newcommand{\calL}{{\cal L}}

\newcommand{\calR}{{\cal R}}

\newcommand{\calU}{{\cal U}}

\newcommand{\calW}{{\cal W}}
\newcommand{\calX}{{\cal X}}

\newcommand{\utld}{{\tilde{u}}}
\newcommand{\vtld}{{\tilde{v}}}

\newcommand{\bfLtld}{\tilde{\bfL}}

\newcommand{\Khat}{{\hat{K}}}
\newcommand{\fhat}{{\hat{f}}}

\newcommand{\bfxtld}{{\tilde{\bfx}}}

\newcommand{\khat}{{\hat{\bf k}}}

\newcommand{\usmhat}{{\hat{ u}}}

\newcommand{\dbfL}{{{\vartriangle}\bf L}}

\newcommand{\bbK}{{\mathbb{K}}}

\newcommand{\bbL}{{\mathbb{L}}}

\newcommand{\chihat}{{\hat{\chi}}}

\newcommand{\bbQ}{{\mathbb{Q}}}

\newcommand{\diag}{{{\rm\, diag}}}

\newcommand{\dbfA}{{{\vartriangle\!\!\bfA}}}

\newcommand{\intbar}{{\int\!\!\!\!\!- }} 
\newcommand{\inttbar}{{\int\!\!\!\!\!\!- }} 
\newcommand{\ssint}{\hspace{-3mm}\int} 

\newcommand{\intY}{\ \ \ \ssint_Y}

\newcommand{\eps}{{\varepsilon}}

\newcommand{\phihat}{\hat{\phi}}

\newcommand{\half}{\frac{\;1}{\;2}}
\newcommand{\diverg} {{\rm div}}
\newcommand{\dx}{d {\bf x}}
\newcommand{\dy}{d {\bf y}}

\newcommand{\aae}{{{\rm a.e.}}}
\newcommand{\aand}{{{\rm\;\; and\;\;}}}

\newcommand{\Tr}{{{\rm Tr}}}
\newcommand{\iif}{{{\rm if}}}

\newcommand{\iin}{{{\rm \;in\;}}}
\newcommand{\oon}{{{\rm \;on\;}}}
\newcommand{\oor}{{{\rm or}}}

\newcommand{\jumpl}{{{{[\![}}}}
\newcommand{\jumpr}{{{{]\!]}}}}

\newcommand{\bibfont}{\small}
\bibpunct{[}{]}{,}{n}{,}{,}
\begin{document}

\title{ New extremal inclusions and their applications to two-phase composites}

\author{ Liping Liu \and Richard  James \and Perry Leo }

\date{Received: date / Revised version: date}

\maketitle

\begin{center}
Draft: {August 17, 2010}
\end{center}

\vspace{-7cm}

{\small Article submitted to Archive for Rational Mechanics and Analysis.} \hfill

\vspace{7cm}

\begin{center}
\section*{Abstract}
\end{center}
In this paper, we find a class of special inclusions that
have the same property with respect to second order linear
partial differential equations as holds for
ellipsoids. That is, in the simplest case and in physical terms,
constant magnetization of the inclusion implies constant
magnetic field on the inclusion.
The special inclusions are found as solutions of a simple variational inequality.
 This variational inequality allows us to
prescribe the connectivity and periodicity properties of
the inclusions.  For example we find periodic arrays of inclusions in two and three dimensions
for which constant magnetization of the inclusions implies constant
magnetic field on the inclusions.
The volume fraction of the inclusions can be any number between zero and one. We find
such inclusions with any finite number of components
and components that are multiply connected.
  These special inclusions enjoy many useful
properties with respect to homogenization and energy
minimization.  For example, we use them to give new results
on a) the effective properties of two-phase composites and
 b) optimal bounds and optimal microstructures for
 two-phase composites.

\vspace{2mm}

\tableofcontents

\renewcommand{\theequation}{\thesection-\arabic{equation}}
\setcounter{equation}{0}
\section{Introduction}\label{sec:intro}

{\sc Poisson}~\cite{Poisson1826} found a remarkable property of ellipsoids: given
a uniformly magnetized ellipsoid,  the
induced magnetic field is uniform inside the
ellipsoid.   Explicit  expressions for this field were
obtained by {\sc Maxwell}~\cite{Maxwell1873}.
{\sc Eshelby}~\cite{Eshelby1957, Eshelby1961} exploited a
similar result in the linear theory of elasticity.
Eshelby's solution asserts that a uniform eigenstrain on an
ellipsoidal subregion in an infinite elastic medium induces
uniform stress inside the ellipsoid (see also, {\sc Mura}~\cite{Mura1987}).
In this paper we find other shapes besides ellipsoids
with these properties.

The relevant problem can be formulated as
the following system of partial differential equations
for $\bfv: \rz^n \to \rz^m$:
\beqs \label{problem:magel}
\diverg[\bfL \nabla \bfv +\bfP\chi_\Omega ]=0\qquad \oon\;\;\rz^n,
\eeqs understood in the sense of distributions. Here,
$\bfP\in \rz^{m\times n}$,  the tensor $\bfL:\rz^{m\times n}\to \rz^{m\times n}$
is assumed to be symmetric and  positive semi-definite,
 and $\chi_{\Omega}$ is the characteristic function of $\Omega \subset \rz^n$.
 $\Omega$ is called the inclusion.
 In the application to ferromagnetism, $m = 1$, $\nabla \bfv$ is the magnetic field,
$\bfP$ is the magnetization, and $\bfL$ is the $n\times n$ identity matrix.  In applications
 to electrostatics, $m = 1$, $\nabla \bfv$ is the electric field,
$\bfP$ is the polarization, and $\bfL$ is
the permittivity tensor of free space.
In the application to linearized elasticity,  $m = n$,  $\nabla \bfv$ is the displacement
gradient, $\bfP$ is the eigenstress and $\bfL$ is the elasticity tensor.

   In  the theory of composites,  in fracture mechanics and in the modeling of phase transformations
a related problem, called the {\em inhomogeneous}
 Eshelby inclusion problem, appears often.
The governing equations for this problem are
\beqs \label{problem:inhomrz}
 \diverg[\bfL(\bfx, \Omega)(\nabla
\bfv(\bfx)+\bfF \chi_{\Omega})]=0\qquad \oon\;\rz^n,
 \eeqs
 where  $\bfF\in \rz^{m\times n}$ is the eigenstrain and the elasticity
tensor
\beq \label{eq:bfLXI}
\bfL(\bfx,\Omega)=\left\{
\begin{array}{cl}
  \bfL_1   & \qquad \bfx\in \Omega,  \\
  \bfL_0   & \qquad \bfx \in \rz^n \setminus \Omega.
 \end{array} \right.
 \eeq
The inhomogeneous Eshelby inclusion problem concerns two
different elastic materials, one inside the inclusion, and
an imposed eigenstrain $\bfF$ on the inclusion.
{\sc Eshelby}~\cite{Eshelby1957} realized that
under suitable mild hypotheses on the elasticity tensors
the homogeneous problem~\eqref{problem:magel} can be used to solve
inhomogeneous problem~\eqref{problem:inhomrz} provided that
the induced field $\nabla \bfv$ for the homogeneous problem~\eqref{problem:magel}
is {\em constant} on  $\Omega$.  It is this relation between (\ref{problem:inhomrz})
and \eqref{problem:magel} that allows us to use the special inclusions
to obtain optimal bounds for composites.

The requirement that a solution $\bfv$
 of problem~\eqref{problem:magel} satisfy $\nabla \bfv =
 constant$ on  $\Omega$ places strong restrictions on the region
$\Omega$. The main purpose of this paper is to find special
regions with this  property in the periodic and
other cases, in all dimensions $n \ge 2$, including cases in which $\Omega$ is
multiply connected.

Since  $\nabla \bfv$ being constant on
$\Omega$ leads to great simplification, ellipsoids play a central
 role in the theory of composites ({\sc Christensen~\cite{Christensen1979}; Milton~\cite{Milton2002}}),
in micromechanics ({\sc Mura~\cite{Mura1987}}) and in experimental measurements
({\sc Brown~\cite{Brown1962}}).  The uniformity of the induced
field can also be used to simplify
 free energy minimization problems that arise
in theories of  ferroelectric and magnetostrictive materials
({\sc Desimone \& James~\cite{DeSimoneJames2002};
 Bhattacharya \& Li~\cite{BhattLi2001};  Liu, James \& Leo~\cite{LiuJamesLeo2006}}),
 and this was our original motivation for developing
the theory of E-inclusions.   Roughly speaking, even though these
are non-convex variational problems, the special properties
of ellipsoids are used to show that certain weak limits of
minimizing sequences are uniform, and this allows one to find,
and also to minimize, the relaxed energy.
Two or more ellipsoids
do not enjoy this special property.  Thus
in many of these applications, only one
ellipsoid can be present in the model, and therefore the
results apply either to isolated ellipsoids or to
composites in the dilute limit.   In most if not all of these
cases it is not an ellipsoid {\it per se} that is being used
but only its property of having  uniform field, when it
is uniformly magnetized. Many authors have speculated on the
possibility that other regions may have this
property ({\sc Mura~\cite{Mura2000}}), and
that, if so, their analysis would also apply to those regions.

We now define an {\bf E-inclusion}, which is the mathematically
natural generalization of an ellipsoid in this context\footnote{ The
terminology ``E-inclusions'' refers to three associations:
 1) this study was motivated by the Eshelby
 inclusion problem as described above;
 2) they are a
generalization of ellipsoids, and, conversely, ellipsoids can
 be regarded as the dilute limits of special periodic E-inclusions;
 3) they are extremal microstructures for a broad class of
 energy minimization problems
 for multiphase composites, as indicted in the title.
 } (see {\sc Liu, James \& Leo~\cite{LiuJamesLeo2007}}).
 We note that, while these definitions concern the scalar-valued case, we will
 show (Section 4) that they apply to many vector-valued examples of the
 type described above. Separate but closely related definitions
 of E-inclusions are given for the periodic and nonperiodic
 cases. More general situations are discussed in Section~\ref{sec:discuss}.

 For the purpose of the definition below define
 \beqs \label{eq:calH}
 \calH = \bigg\{v\in
W^{1,2}_{loc}(\rz^n):
\; \limsup_{|\bfx| \to \infty} |\bfx|^{n-2}|v(\bfx)|< \infty &\iif\;n\ge 3,\\
\limsup_{|\bfx| \to \infty} (\log|\bfx|)^{-1}|v(\bfx)|< \infty&\iif\;n=2  \bigg \}.
 \nonumber \eeqs
By elementary calculations using the Newtonian potential, the equation
$\Delta u = \chi$ on $\rz^n$ with $\chi \in L^{\infty}$, supp$\, \chi$ compact,
has a unique
solution $u \in \calH$ (for $n= 2$, unique up to an additive constant,
resp.) in the sense of distributions.
It is  known from elliptic regularity theory that
$u \in W^{2,2}_{loc}(\rz^n)$.

\begin{definition} \label{def:Ein}
Let  $\Omega_i$  ($i=1,\cdots, N$) be open, bounded
 and mutually disjoint subsets of $\rz^n$, and
 $\bbK=(\bfQ_1,\cdots, \bfQ_N)$ be an array of $N$ symmetric $n \times n$
 matrices.
\renewcommand{\labelenumi}{\roman{enumi})}
\begin{enumerate}

 \item \;   $\Omega=\cup_{i=1}^N\Omega_i$ is an
   {\bf E-inclusion} corresponding to $\bbK$
  if the solution $u \in \calH$ of
\beqs \label{sec1:over1Prz}
     \Delta u=
\sum_{i=1}^N p_i\chi_{\Omega_i}\quad \oon\;\rz^n, \ \
p_i=\Tr(\bfQ_i) \eeqs
 satisfies
 \beqs \label{eq:overnonp}
\nabla\nabla u=\bfQ_i&\;\;  \aae\;\oon\; \Omega_i \;
\;\qquad\forall\;i=1,\cdots, N. \eeqs


\item Let a Bravais lattice
$\calL = \{\sum_{i=1}^n \nu_i \bfe_i: \nu_i \in \gz\;\aand\;
\bfe_1,\cdots, \bfe_n\in \rz^n \; \mbox{ are linearly independent}
\}$ with an open unit cell $Y = \{\sum_{i = 1}^{n} x_i \bfe_i: 0 <
x_i <1 \}$ be given. Assume $\Omega=\cup_{i=1}^N\Omega_i \subset Y$.
Then  the  set $\Omega_{per}=\cup_{\bfr\in \calL} \{\bfr+\Omega\}$
is a {\bf periodic E-inclusion}  corresponding to $\bbK$  if  a weak
solution  of \beqs \label{sec1:over1P} \becs
     \Delta u= \sum_{i=0}^N p_i\chi_{\Omega_i}&\quad \oon\;Y, \quad p_i=\Tr(\bfQ_i), \\
\mbox{periodic boundary conditions} &\quad \oon\;\partial Y\\
\eecs \eeqs in $W^{2,2}_{per}(Y)$ satisfies \beqs \label{eq:overp}
\nabla\nabla u=\bfQ_i& \quad\aae\;\oon\; \Omega_i\;
\qquad\;\forall\;i=1,\cdots, N, \eeqs
 where $\Omega_0=Y\setminus \Omega$,
 and $p_0\in \rz$ is such that $ \sum_{i=0}^N p_i\theta_i=0$. Here
  $\theta_i=|\Omega_i|/|Y|$ is the volume fraction of $\Omega_i$ in $Y$.
\end{enumerate}
\end{definition}


 Equations~\eqref{problem:magel} and \eqref{problem:inhomrz} and
their periodic counterparts are
 closely related to equations~\eqref{sec1:over1Prz}
 and \eqref{sec1:over1P}, respectively.  In particular,
in the magnetic case (recall $m=1$), $\bfv = \bfP \cdot \nabla u$, and the magnetization
on inclusion $\Omega_i$ ($i = 1, \dots, N$) is $(p_i-p_0) \bfP$ and zero elsewhere. Also,
each matrix in $(\bfQ_1,\cdots, \bfQ_N$) plays
the same role as the conventional {\em demagnetization matrix} does
for an isolated ellipsoid. The demagnetization matrix defines
the linear transformation that maps the magnetization
to the magnetic field  when $\Omega$ is an ellipsoid.
Further discussion of the relation between
\eqref{problem:magel}, \eqref{problem:inhomrz} and
\eqref{sec1:over1Prz}, \eqref{sec1:over1P}
is given in Section \ref{sec:hominhom}.

Besides ellipsoids,  examples of E-inclusions
include the well-known construction of
{\sc Vigdergauz~\cite{ Vigdergauz1986}} for two dimensional periodic E-inclusions
which are simply connected in one unit cell.
Other examples of which we are aware include two dimensional
two-component E-inclusions in a recent paper of
{\sc Kang, Kim \& Milton~\cite{KangMilton2006b}}.
 While considering the effective properties of an elastic plate with
a periodic array of ``equal-strength'' holes,
Vigdergauz found his construction using
complex variable methods. {\sc Grabovsky \& Kohn}~\cite{GK1995b}
gave a concise derivation of the
Vigdergauz construction
and showed that it is an optimal microstructure for two-phase composites.
These results on the status of E-inclusions as optimal
structures in homogenization theory, and the results
mentioned above on energy minimization, show that
E-inclusions have a more fundamental relation to the
underlying equations than merely as a method of
simplification.  We present several applications of this
type  in Section~\ref{sec:appl}.

In Section~\ref{sec:fbvss} we present a general method for constructing  E-inclusions. They
 are found as solutions of a simple variational inequality with a piecewise
quadratic obstacle. The  region where the
minimizer touches the obstacle defines
the  E-inclusion. The existence and regularity of minimizers for this variational inequality
are adapted from the textbooks of {\sc  Kinderlehrer \& Stampacchia~\cite{KindStamp1980}; Friedman~\cite{Friedman1982}}.
 We further show that periodic E-inclusions corresponding
to a single negative-semidefinite matrix ($N=1$) can be
constructed in all dimensions and with  any volume fraction.
The lattice vectors defining the periodicity can be arbitrarily prescribed.
In  two dimensions, one family of these periodic
E-inclusions  specialize to the Vigdergauz microstructure. By varying
the piecewise quadratic obstacle in the variational inequality,
a variety of more general  E-inclusions  can be produced.
For example, we can construct periodic E-inclusions with multiply connected components.
 In Section~\ref{sec:example} we present a numerical scheme for calculating periodic E-inclusions
and various examples of the calculated  periodic E-inclusions.
 In Section~\ref{sec:hominhom} we solve  the periodic Eshelby inclusion problem and explicitly
 calculate the effective properties of two-phase composites with periodic
 E-inclusion structures.
In Section~\ref{sec:locminI} explicit bounds for the effective properties of two-phase composites are derived
and are shown to be attained by the periodic E-inclusions.
In Section~\ref{sec:discuss} we introduce a generalized concept of E-inclusions
called sequential E-inclusions in terms of gradient Young measure.
 We finish with a summary of our results.

\renewcommand{\theequation}{\thesection-\arabic{equation}}
\setcounter{equation}{0}

\section{ A method for constructing special inclusions} \label{sec:fbvss}

\subsection{Existence of periodic E-inclusions}~\label{Sec24}
In this section, we present a method for the construction of
E-inclusions. Readers who are more interested in the examples and
applications of E-inclusions may skip to Sections \ref{sec:example}
and \ref{sec:appl}.

A general method for constructing special inclusions  is based on
a  {\em variational inequality}
({\sc Kinderlehrer \& Stampacchia~\cite{KindStamp1980}; Friedman~\cite{Friedman1982}}).
For periodic E-inclusions in $\rz^n$ ($n\ge 1$), we consider
\beqs
\label{sec1:fbv1}
G_f(u_f)\equiv \min_{u\in K_{per}} \bigg\{G_f(u)\equiv
\int_Y [ \half | \nabla u|^2 +f u ]d \bfx \bigg\}, \eeqs
where  $f> 0$ is a constant and the admissible set $K_{per}=\{u \in
W_{per}^{1,2}(Y)\,:\; u(\bfx) \ge \phi_{per}(\bfx) \;\aae\oon\;Y\}$.
This type of minimization problem is also known as a {\em free-boundary
obstacle problem}  and the given function
$\phi_{per}$ is called the {\em obstacle}. Here and afterwards,
$Y$ is an open unit cell associated with a Bravais lattice $\calL$ as defined
above. We  further assume
\beqs \label{eq:phiper}
\phi_{per}\in C^{0,1}_{per}(Y)\;\;\aand\;\;\hspace{6cm}\nonumber\\
\frac{\partial ^2 \phi_{per}}{\partial \xi^2} \ge -C \mbox{ in the sense of distributions on $\rz^n$},
\eeqs
where $\xi$ is any unit vector in $\rz^n$, $\partial /\partial \xi$ denotes the
directional derivative and $C>0$ is a constant.
This means that for any nonnegative
$\varphi\in C_c^\infty(\rz^n)$ and any unit vector $\xi \in \rz^n$,
\beqs \label{eq:maincond}
\int [\phi_{per}+\half C|\bfx|^2] \frac{\partial^2\varphi}{\partial \xi^2} \dx \ge 0.
\eeqs

The existence and uniqueness of the
solution of the variational inequality~\eqref{sec1:fbv1}
in the periodic case is not available, so a proof is given below.
Once the existence of a minimizer is established, its
regularity  then follows from a
careful reading of {\sc Brezis \& Kinderlehrer}~\cite{BrezisKind1974} (see also
 {\sc Friedman}~\cite{Friedman1982}).
\begin{theorem}~\label{thrm:EUforfbv}
The variational inequality~\eqref{sec1:fbv1}-\eqref{eq:phiper} has a unique
minimizer $u_f\in W^{2,\infty}_{per}(Y)$ for  $f>0$.
\end{theorem}
\begin{proof}
Without loss of generality assume meas$\, Y = 1$.
  Choosing the constant test function $v=M=\sup_{\bfx\in Y}| \phi_{per}(\bfx)|$,
we verify that
\beas |\inf_{v\in K_{per}} G_f(v)| \leq f M<+\infty. \eeas
 Let $v_k\in K_{per}$ be a minimizing sequence
 such that
 \beas \inf_{v\in K_{per}} G_f(v)\leq
G_f(v_k)=\half \intY|\nabla \vtld_k|^2\,\dx+d_k f \leq \inf_{v\in K_{per}}
G_f(v)+\frac{1}{k},
 \eeas where
$d_k=\int_Y v_k \dx$ and $\vtld_k=v_k-d_k\in \calW^{1,2}_{per}(Y)=\{ \varphi\in
W^{1,2}_{per}(Y):\;\int_Y \varphi\dx=0\}$. Clearly, $d_k\geq \int_Y \phi_{per}(\bfx)\dx\ge -M$.
Thus, since $f>0$,
\beas \half \intY|\nabla \vtld_k|^2\,\dx \le 2fM+\frac{1}{k} \quad
{\rm and } \quad |d_k|\le M+\frac{1}{fk} .
\eeas
Therefore, up to a subsequence and
without relabelling, we have \beas \vtld_k \rightharpoonup \utld_f
\mbox{\rm \;weakly in } \calW^{1,2}_{per}(Y),\quad d_k\rightarrow d_f
{\rm\;in\;}\rz,\eeas
and
\beas
\quad v_k=\vtld_k+d_k \rightharpoonup u_f :=\utld_f+d_f\in K_{per}
\mbox{\rm\; weakly in } W^{1,2}_{per}(Y),
\eeas
which, by the lower semi-continuity of the functional $G_f$ with respect to weak convergence
in $W^{1,2}_{per}(Y)$, implies
$u_f$  is a minimizer.

If there is another minimizer $v_f \in K_{per}$,
from the convexity of the set $K_{per}$, we see that
$w_\eps=u_f+\eps (v_f-u_f)\in K_{per}$ for all $\eps\in (0,1)$.
 Since $u_f$ being a minimizer implies $\frac{1}{\eps}
[G_f(w_\eps)-G_f(u_f)] \ge 0 $, sending $\eps$ to $0$ we obtain
\beqs \label{sec1:minweak} \intY
[\nabla u_f \cdot\nabla (v_f-u_f) +f(v_f-u_f)]\dx \geq 0.
\eeqs
Similar calculations applied to the minimizer $v_f$ and
$w_\eps=v_f+\eps (u_f-v_f)$  yield
\beqs \label{sec1:minweak1} \intY
[\nabla v_f \cdot\nabla (u_f-v_f) +f(u_f-v_f)]\dx \geq 0.
\eeqs
Adding \eqref{sec1:minweak} and \eqref{sec1:minweak1}, we obtain
 $\int_Y|\nabla u_f- \nabla v_f|^2\dx\le 0$. Thus, $u_f$ and $v_f$ can be different at most by a constant.
But $G_f(u_f)=G_f(v_f)$ and $f>0$  imply this constant must be zero, i.e.,
$v_f=u_f$. The statement of regularity $u_f \in W^{2 \infty}(Y)$ follows
from {\sc Brezis \& Kinderlehrer}~\cite{BrezisKind1974}, as noted above.
\end{proof}

By the result ($u_f\in W^{2,\infty}_{per}(Y)$), the following {\em
complementarity conditions} hold (see {\sc Duvaut \& Lions}~\cite{DuvautLions1977}):
\beqs \label{eq:Deltauf}
-\Delta u_f+f\ge 0, \qquad u_f\ge \phi_{per},\;\;\aand \nonumber \\
\qquad (-\Delta u_f+f)(u_f-\phi_{per})=0 \;\;\;\qquad\aae\; \oon\; Y.
\eeqs
 Once the $W^{2,\infty}$ regularity  is established, the
conditions (\ref{eq:Deltauf}) are equivalent to the variational inequality.

Recalling the definition of periodic E-inclusions in Section~\ref{sec:intro},
we use  periodic piecewise
quadratic obstacles to construct periodic E-inclusions.  First, we assign $N$
quadratic functions $q_1,\dots, q_M$ on $Y$.  Second, we consider a disjoint measurable subdivision
of $Y$ into subsets $\calU_1,\dots, \calU_M$.  We say that $\phi_{per}:\rz^n \to \rz$ is a
{\bf periodic piecewise quadratic obstacle}
if
\renewcommand{\labelenumi}{(\roman{enumi})}
\begin{enumerate}
\item $\phi_{per}$ satisfies  \eqref{eq:phiper}, and
\item $\phi_{per} = q_{i}$ on $\calU_{i}$ for $i=1,\cdots, M$.
\end{enumerate}

 In general one may not be able to construct a periodic
 piecewise quadratic obstacle from given quadratic functions.
Below we give
 two examples of periodic  piecewise quadratic obstacles.
Both examples use concave quadratic functions.
 We note that the definition, however, also includes some cases in which some
 of the $q_i$ are convex and others are concave\footnote{Specific examples
 are easily constructed.}.
  The latter is important for extending the attainability of
the Hashin-Shtrikman bounds for multiphase composites ({\sc Liu}~\cite{LiuMultiphase, LiuPHL}).

\begin{example} Let  $q_1,\dots, q_M$  be strictly concave
quadratic functions {\it defined on} $\rz^n$, and let $\calL$ be as in
the definition of a periodic E-inclusion.
 Then
\beqs \label{eq:phiperexample}
\phi_{ per}(\bfx) = \sup \{ q_i(\bfx + \bfr): \bfr \in \calL, i = 1, \dots, M \}
\eeqs
is a periodic piecewise quadratic obstacle. To show this,  first note
that $\phi_{per}\in C^{0,1}_{per}(Y)$ because, restricted to $Y$,
(\ref{eq:phiperexample}) is the sup over a finite number of quadratic
functions.  Let $C>0$ be
a constant large enough such that $C\bfI+\nabla \nabla q_i$
is positive definite for all $i=1,\cdots, M$, where $\bfI$ is
the $n\times n$ identity matrix.
Change variables if necessary, it suffices to verify \eqref{eq:maincond}
for direction $\xi=(1,0\cdots, 0)$.

For any  nonnegative $\varphi\in C_c^\infty(\rz^n)$,
let $(a,b)\times D \subset \rz^n$ be an open finite rectangular box containing the support of
 $\varphi$.
Restricted to this box, $\phi_{per}$ defined in \eqref{eq:phiperexample} can
again be expressed as the supremum of finite many quadratic functions
$q'_j$ ($j=1,\cdots, M'$) which are translates of $q_1,\cdots, q_M$, and so
\beqs \label{eq:phiperD}
\phi_{ per}(\bfx)+\half C|\bfx|^2 =\hspace{6cm}\\
 \sup \{ q'_j(\bfx )+\half C|\bfx|^2:  j = 1, \dots, M' \} \;\qquad \oon\;(a,b)\times D. \nonumber
\eeqs By Fubini's theorem, we have \beqs \label{eq:Fubini} \int
[\phi_{per}+\half C|\bfx|^2] \frac{\partial^2\varphi}{\partial
x_1^2} \dx =\int_{D} \int_a^b [\phi_{per}+\half C|\bfx|^2]
\frac{\partial^2\varphi}{\partial x_1^2}dx_1 d \bfxtld, \eeqs where
$\bfxtld=(x_2,\cdots, x_n)$. For fixed $\bfxtld$, let
$g(x_1)=\phi_{per}+\half C|\bfx|^2$. It is clear that $g$ is
continuous and $g'(x_1)=dg(x_1)/dx_1$ has only finitely many
discontinuities in the interval $(a, b)$, enumerated as
$x_1^\ast<\cdots < x_m^\ast$.

At a discontinuous point of $g'(x_1)$, e.g., $x_1^\ast$, by
\eqref{eq:phiperD} we see that for $\eps>0$ small enough, \beas
g(x_1)= \becs
s_1(x_1)>s_2(x_1)&\iif\;x_1\in (-\epsilon+x_1^\ast,x_1^\ast),\\
s_2(x_1)>s_1(x_1)&\iif\;x_1\in (x_1^\ast,\epsilon+x_1^\ast),\\
\eecs \eeas and $s_1(x_1^\ast)=s_2(x_1^\ast)$, where $s_1(x_1)$ and
$s_2(x_1)$ are two quadratic functions. Thus, \beas \jumpl
g'(x_1^\ast)\jumpr= g'(x_1^\ast+)-g'(x_1^\ast-)=s_2'(x_1^\ast)-
s_1'(x_1^\ast)\ge 0. \eeas
Integrating by parts for the $m$ discontinuities,
we have \beqs \label{eq:varphimu}
 &&\int_a^b
[\phi_{per}+\half C|\bfx|^2] \frac{\partial^2\varphi}{\partial
x_1^2}dx_1= -\int_a^b\frac{\partial \varphi}{\partial x_1} g'(x_1)dx_1  \\
&&\hspace{1cm}=\sum_{i=0}^{m} \int_{x_i^\ast}^{x_{i+1}^\ast}\varphi
g_1''(x_1)dx_1+\sum_{i=1}^m \jumpl g(x_i^\ast)\jumpr \varphi
(x_1=x_i^\ast,\bfxtld) \ge 0  \nonumber \eeqs for any nonnegative $\varphi$,
 where $x_0^\ast=a$ and  $x_{m+1}^\ast=b.$ Since
\eqref{eq:varphimu} holds for any $\bfxtld \in D$,
equation~\eqref{eq:maincond} follows from \eqref{eq:Fubini} and
\eqref{eq:varphimu}.

\end{example}

\begin{example} In this example we consider one inclusion per unit cell
but allow some eigenvalues of $\bfQ_1$ to be zero.
Assume $N = 1$ and consider a negative semi-definite symmetric
matrix $\bfQ_1 = \bfQ\neq 0$ and
denote by $\calR(\bfQ)\subset \rz^n$ the range of  $\bfQ$.
Let $(\bfe_1,\cdots, \bfe_{n'})$ be a basis of the subspace $\calR(\bfQ)$.
Then
\beqs \label{eq:phiperexample2}
\phi_{ per}(\bfx) = \sup \{
\half (\bfx + \bfr)\cdot\bfQ(\bfx + \bfr):
 \bfr=\sum_{i=1}^{n'} \nu_i\bfe_i, \;\nu_i\in \gz\}
\eeqs
is a periodic piecewise quadratic obstacle. The proof that this construction gives
a periodic piecewise quadratic obstacle is similar to that in Example~1.
\end{example}

Having constructed the obstacle, we now turn to the theory of
variational inequalities.  Let $u_f$ be the solution given in Theorem
\ref{thrm:EUforfbv} of the variational
inequality corresponding to a piecewise quadratic periodic
obstacle $\phi_{per}$.
 The coincident set $\Omega^f_{per}$ and
noncoincident set $N^f_{per}$ are defined by
\beqs \label{sec1:Oper}
\Omega^f_{per}:= \{\bfx\in
\rz^n:\:u_f(\bfx)=\phi_{per}(\bfx) \}
\eeqs
and
\beqs \label{sec1:Nper}
N^f_{per}:=\{\bfx\in \rz^n:\:u_f(\bfx)>\phi_{per}(\bfx) \},
\eeqs
 respectively.
 The definition of the coincident set clearly implies
\beqs \label{eq:nablanablau}
\nabla \nabla u_f(\bfx)=\nabla\nabla \phi_{per}(\bfx)
 \ \ \aae \ \ {\rm on }\  \Omega_{per}^f.
\eeqs
Therefore, the minimizer $u_f\in W^{2,\infty}_{per}(Y)$ solves the  overdetermined problem
 \beqs \label{problem:over3} \becs
\Delta u_f=f \chi_{N^f_{per}} +\Delta \phi_{per} \chi_{\Omega^f_{per}}  &\aae\;\oon\;Y,\\
\nabla \nabla u_f(\bfx)=\nabla\nabla \phi_{per}(\bfx) & \aae\ \
 {\rm on }\ \  \Omega^f_{per}\cap Y,\\
\mbox{periodic boundary conditions} & \oon\;\partial Y.\\
\eecs \qquad \eeqs
Let $\bbK = \{ \nabla \nabla q_i: i = 1, \dots, M \} = \{\bfQ_1, \dots, \bfQ_N\}$,
the $\bfQ_i$ being distinct symmetric $n \times n$ matrices.
Let $\Omega_i \subset \Omega_{per}^f$ be the largest open
set such that
\beq \label{ei}
\int_{\Omega_i} |\nabla \nabla u_f - \bfQ_i|^2\, d\bfx = 0.
\eeq
Clearly, the $\Omega_1, \dots, \Omega_N$ are disjoint and open. We have that
\beqs \label{eq:Ein}
\Delta u_f=f\;\;& \aae \oon&\;\Omega_0=Y\setminus (\cup_{i=1}^N\Omega_i)\qquad\aand \\
\qquad
\nabla\nabla u_f(\bfx)=\bfQ_i\;\;\;& \aae \oon&\;\Omega_i
\qquad \;\;\qquad \forall\,i=1,\cdots, N.\nonumber
\eeqs
Therefore, we have obtained the following result.
\begin{theorem} \label{thrm:exist}
 Consider the variational inequality~\eqref{sec1:fbv1} with
  a periodic piecewise quadratic obstacle  $\phi_{per}$.
  Then the periodic extension of $\Omega = \cup_1^N \Omega_i$,
  with $\Omega_i$ defined by (\ref{ei}),
  is a periodic E-inclusion
  corresponding to $\bbK=(\bfQ_1,\cdots, \bfQ_N)$  and
  with $p_0 = f$, for any $f>0$.
\end{theorem}

We now discuss restrictions on the volume fractions of periodic
E-inclusions.  Recall that  $\theta_i=|\Omega_i|/|Y|$ is
the volume fraction of $\Omega_i$
 in $Y$ ($i=0,1,\cdots, N$).
 Clearly,  the volume fractions $\Theta = \{ \theta_1, \dots, \theta_N \}$
 necessarily satisfy
 \beqs \label{eq:thetaconstraints}
\theta_i\in [0,\;1]\;\mbox{for all }i=1,\cdots, N
\aand\;1-\theta_0=\sum_{i=1}^N \theta_i\in (0, \;1).
 \eeqs
However, they are  not all known {\em a priori}. Additionally,
they  satisfy
\beqs \label{eq:fvf}
\inttbar_Y \Delta u_f \dx=0\;\;\Longrightarrow\;\;
f\theta_0+\sum_{i=1}^N p_i \theta_i=0,
\eeqs
where $\intbar_V\;\;$ denotes the average of the integrand over $V$,  $p_i=\Tr(\bfQ_i)$ ($i=1,\cdots, N$) from the second of \eqref{eq:Ein}.
Since  $f$ can be any positive number, equation~\eqref{eq:fvf} implies that
the volume fraction  $1-\theta_0$ of the periodic E-inclusion
  can be any number between zero and one in the case $N=1$.

There are non-obvious restrictions on  $\bbK$ and $\Theta$ that arise from
the definition of a periodic E-inclusion.
Let $u$ be the solution of  \eqref{sec1:over1P}-\eqref{eq:overp} appearing
in the definition of a periodic E-inclusion.
For any $\bfm\in \rz^n$,  the divergence theorem implies that
 \beqs \label{eq:bfQi}
 \theta_0 \inttbar_{\Omega_0} |(\nabla \nabla u) \bfm|^2 \dx \hspace{5cm} \nonumber \\
 = \bfm \cdot [\inttbar_Y \Delta u \nabla\nabla u\dx]\bfm-
\sum_{i=1}^N \theta_i \inttbar_{\Omega_i} |(\nabla \nabla u) \bfm|^2 \dx.
 \eeqs
We bound the left-hand side of \eqref{eq:bfQi} using Jensen's inequality:
\beqs \label{ine}
\theta_0 \inttbar_{\Omega_0} |(\nabla \nabla u) \bfm|^2 \dx &\ge&
\theta_0 \bfm\cdot\bigg[\inttbar_{\Omega_0} \nabla \nabla u \dx\bigg]^2\bfm \nonumber \\
 &=& \frac{1}{\theta_0} \bfm\cdot\bigg[\sum_{i=1}^N \theta_i\inttbar_{\Omega_i}\nabla\nabla u\dx \bigg]^2\bfm.
\eeqs
The last step in (\ref{ine}) follows from the periodicity of $u$:
\beqs \label{eq:upp0}
\inttbar_Y \nabla \nabla u \dx=0 \ \Longrightarrow \
\theta_0 \inttbar_{\Omega_0} \nabla \nabla u \dx
 = - \sum_{i=1}^N \theta_i\inttbar_{\Omega_i}\nabla\nabla u\dx.\qquad
 \eeqs
For the first term on the right-hand side of \eqref{eq:bfQi}, since $\Delta u = p_0$ on $\Omega_0$,
we have
\beas
\theta_0 \bfm \cdot [\inttbar_{\Omega_0} \Delta u \nabla\nabla u\dx]\bfm
&=&p_0 \bfm \cdot [ \theta_0\inttbar_{\Omega_0}\nabla\nabla u\dx]\bfm\\
&=&p_0 \bfm \cdot [-\sum_{i=1}^N \theta_i\inttbar_{\Omega_i}\nabla\nabla u\dx]\bfm,
\eeas
where the second equality is justified by using again \eqref{eq:upp0}. Therefore,
equation~\eqref{eq:bfQi} implies the following restriction on $\bbK$ and $\Theta$\,:
\beqs \label{eq:QiNeccond}
\sum_{i=1}^N \bigg[\theta_0\Tr(\bfQ_i)+\sum_{j=1}^N \theta_j\Tr(\bfQ_j)\bigg]
\theta_i\bfQ_i\ge
\theta_0\sum_{i=1}^N \theta_i\bfQ_i^2+\bigg[\sum_{i=1}^N \theta_i\bfQ_i \bigg]^2, \qquad
\eeqs
where equations~\eqref{eq:overp} and $\theta_0 p_0=-\sum_{j=1}^N \theta_j\Tr(\bfQ_j)$ have been used.
Also, for two symmetric tensors, $\bfM_1\ge\,({\rm resp., }>)\, \bfM_2$ means $\bfM_1-\bfM_2$
is positive semi-definite (resp., positive definite). This convention is followed subsequently.
It is not known in general whether all $\bbK$ and
$\Theta$ satisfying~\eqref{eq:thetaconstraints} and \eqref{eq:QiNeccond}
can be achieved by  periodic E-inclusions.
For many applications, the following question is crucial:

\vspace{2mm}
\noindent{\bf Question 1.} {\em
For what values of $\bbK$ and $\Theta$ can we find a periodic E-inclusion?}
\label{page:question1}

\vspace{2mm}

For some special cases, the answer to Question 1 is known. The
following remark describes such an example.

\begin{remark}\label{rmk:thetaQ}
 In the case  $N=1$, equations~\eqref{eq:thetaconstraints} and \eqref{eq:QiNeccond}
 are equivalent to
 \beqs \label{eq:thetaQ}
\theta\in (0,1)\qquad\aand\qquad \bfQ\ge 0\; \oor \;\bfQ\le 0,
 \eeqs
 where $\theta$ is the volume fraction of the periodic E-inclusion
 constructed from Theorem \ref{thrm:exist} based on Example 2.
 We have suppressed the subscript ``$_1$''.\;
By  $\intbar_Y\Delta u_f \, d\bfx =0 $, we have $\theta=f/(f-\Tr(\bfQ))$.
 Since $f$ can be any positive number,
 $\theta$ can be any number between zero and one.

By replacing $u$ by $a u$  ($a\in \rz$) in the Definition~\ref{def:Ein},
we see that if $\Omega_{per}$ is
a periodic E-inclusion corresponding to $\bfQ$  and $\theta$,
then it  is also a periodic E-inclusion corresponding to $a\bfQ$  and $\theta$.
The case $\bfQ\ge 0$ follows from the case $\bfQ\le 0$ by setting $a=-1$.
\end{remark}

It is often desirable to fix the arbitrary multiplicative constant
associated with the matrix $\bfQ \ne 0$.
 For future convenience, let us rephrase   Remark~\ref{rmk:thetaQ}
as the following theorem.
\begin{theorem}
Let
\label{thrm:QEin}
\beqs \label{eq:bbQ}
\bbQ:=\{X\in \rz^{n\times n}_{sym}:\;X \ge 0 \,\aand \;\Tr(X)=1 \} .
 \eeqs
For any matrix $\bfQ\in \bbQ$ and any
volume fraction $\theta\in (0,\,1)$,
there exists a periodic E-inclusion $\Omega_{per}$, i.e., there is
$u \in  W^{2,2}_{per}(Y)$ satisfying
\beqs \label{problem:over1norm}
\becs
\Delta u=\theta-\chi_{\Omega} &\aae\oon\;Y\\
\nabla\nabla u=-(1-\theta)\bfQ &\aae\oon\;\Omega\\
\mbox{\rm periodic boundary conditions} & \oon\;\partial Y\\
\eecs
\eeqs
where
  $\Omega=Y\cap \Omega_{per}$ and $\theta=|\Omega|/|Y|$.
 Conversely, if the overdetermined problem~\eqref{problem:over1norm}
 has a weak solution $u \in W^{2,2}_{per}(Y)$ for a nonzero matrix
 $\bfQ$, then the matrix $\bfQ$ must belong to $\bbQ$.
\end{theorem}

\noindent Proof.  Only the last statement needs proof, but this follows
immediately by taking the trace of the second equation in
\eqref{problem:over1norm}, and also by using the inequality
\eqref{eq:QiNeccond}.


\subsection{Existence of nonperiodic E-inclusions for $n \ge 3$}

To construct nonperiodic E-inclusions in $\rz^n,\, n\ge 3$,  we use
 the  variational inequality
\beqs \label{eq:fbv1}
G_r(u_r)=\inf_{u\in K_r} \bigg\{G_r(u)\equiv \int_{B_r} \half | \nabla u|^2  d \bfx \bigg\},
\eeqs
where  $B_r$ is the open ball centered at the origin of radius $r$, and
the admissible set is \beqs \label{eq:Kd} K_r=\{v
\in W_0^{1,2}(B_r)\,:\;v\ge \phi \}.
 \eeqs
We use this variational inequality to find $u_r$ and the coincident set,
and then we pass to the limit $r \to \infty$ to establish the existence
of nonperiodic E-inclusions.

Let $q_1,\dots, q_N$ be quadratic functions on $\rz^n$ and  $\calU_1,\dots, \calU_N$
 a disjoint measurable subdivision of $\rz^n$.
As before, we say that $\phi:\rz^n \to \rz$ is a {\bf piecewise quadratic obstacle}
if
\renewcommand{\labelenumi}{(\roman{enumi})}
\begin{enumerate}
\item $\phi\in C^{0,1}(\rz^n)$, $\partial^2\phi/\partial \xi^2>-C$ on $\rz^n$ in the
distributional sense, for all $|\xi| = 1$,
\item $\phi = q_{i}$ on $\calU_{i}$, $i=1,\cdots, N$, and
\item $\phi$ is bounded from above on $\rz^n$, $\max\phi:=\max_{\bfx\in \rz^n} \phi(\bfx)>0$,
and there exists $R_0 >0$ such that $\phi(\bfx)<0 \;\;
\iif\;\;|\bfx| \ge R_0$.
\end{enumerate}
Here, as above,
$\partial/\partial \xi$ denotes the directional derivative.
From the basic theory
({\sc Brezis \& Kinderlehrer~\cite{BrezisKind1974}};
also {\sc Kinderlehrer \& Stampacchia~\cite{KindStamp1980}};
{\sc Friedman~\cite{Friedman1982}}),  we have
\begin{theorem}~\label{thrm:EUforfbv1}
The variational inequality~\eqref{eq:fbv1} with $\phi$ being a piecewise quadratic
obstacle has a unique
minimizer $u_r\in W^{2,\infty}(B_r)\cap W_0^{1,\infty}(B_r)$ for each $r\ge R_0$.
Further, the unique minimizer satisfies
\begin{enumerate}
  \item  $0\le u_r\le \max \phi$  on $B_r$, and
  \item There exists a constant $C>0$, independent of $r$, such that
\beqs \label{eq:Linftybds}
\| \nabla \nabla u_r\|_{L^\infty(B_r)} <C.
\eeqs
\end{enumerate}
\end{theorem}

Since $u_r$ is a solution of the variational inequality, as we have noted previously,
$u_r$ also satisfies the complementarity conditions
\beqs \label{eq:Deltauf1}
-\Delta u_r\ge 0, \;\; u_r\ge \phi,\;\;\aand\;\; -\Delta u_r(u_r-\phi)=0 \;\;\;\aae\oon\; B_r. \qquad
\eeqs
In particular, letting $\Omega_r$ be  the coincident set,
\beqs \label{eq:uroverd}
\nabla \nabla u_r=\nabla \nabla \phi \qquad\;\;\aae\
 {\rm on}\; \Omega_r.
\eeqs

The limit of the minimizers $u_r$ of problem~\eqref{eq:fbv1}
can be defined as follows.
Let $r_j\to +\infty$ be an increasing sequence.
Let $R> R_0$.  From equation~\eqref{eq:Linftybds}
and $\|u_r\|_{L^\infty(B_R)}<\sup_{B_R}|\phi|$,
 it follows that there is a constant $M$ independent of $r$ such that
\beqs \label{eq:urW2infty}
\|u_r\|_{ W^{2,\infty}(B_R)}\le M.
\eeqs
 Since $u_{r_j}$ is  uniformly bounded in $W^{2,\infty}(B_{R})$,
 there exists $u_\infty\in W^{2,\infty}(B_R)$ such that, up to a subsequence,
\beqs \label{eq:urnweaktoinfty}
u_{r_j} \rightharpoonup u_\infty\; \mbox{ weakly}^\ast\;\;\iin\;\; W^{2,\infty}(B_{R}).
\eeqs
From~\eqref{eq:Deltauf1} and \eqref{eq:urnweaktoinfty}, it is easy to verify that
\beqs \label{eq:Deltauf2}
-\Delta u_\infty\ge 0, \;u_\infty\ge \phi,\;\aand\;
-\Delta u_\infty(u_\infty-\phi)=0 \;\;\aae\;\oon\; B_R. \qquad
\eeqs
In particular, the first two of \eqref{eq:Deltauf2} follow from linearity, while the
third of  \eqref{eq:Deltauf2} is justified by  the uniform convergence of $u_r \to u_{\infty}$.
Since $R$ is arbitrary, $u_\infty\in W^{2,\infty}_{loc}(\rz^n)$ satisfies  \eqref{eq:urnweaktoinfty} and
\eqref{eq:Deltauf2} for any $R> R_0$.

Let $\Omega_{\infty}$ be the coincident set of $u_{\infty}$.
We claim that $u_\infty$ satisfies
\beqs \label{eq:uinftyoverd}
\becs
\Delta u_\infty=\chi_{\Omega_\infty}\Delta \phi\qquad\aae\oon\;\;\rz^n,\\
\nabla \nabla u_\infty=\nabla \nabla \phi\qquad \aae\ {\rm on}\;\; \Omega_\infty ,\\
\limsup_{|\bfx| \to \infty} |\bfx|^{n-2}|u_{\infty}(\bfx)| <  \infty.
\eecs
\eeqs
The first two equations in~\eqref{eq:uinftyoverd} are consequences of
the last equation in \eqref{eq:Deltauf2} and the definition of the
coincident set $\Omega_\infty$.
To justify the last equation,  we consider $u_r \ge 0 $ given by Theorem
\ref{thrm:EUforfbv1}, and extend it to $\rz^n$ by putting $u_r  = 0$
on $\rz^n \setminus B_r$. By testing against a positive test function
supported near $\partial B_r$ we see that $u_r$ is subharmonic on
$\rz^n \setminus B_{R_0}$.  Since $|\bfx|^{2-n}$ is harmonic on
$\rz^n \setminus B_{R_0}$, then
\beq
 w_r(\bfx) = u_r(\bfx) - \max \phi \left(\frac{R_0}{|\bfx|}\right)^{n-2}
\eeq
is also subharmonic on $\rz^n \setminus B_{R_0}$.  In addition $w_r$
is nonpositive on $\partial (\rz^n \setminus B_{R_0})$.  Hence, by
the maximum principle $w_r \le 0$ on $\rz^n \setminus B_{R_0}$.
Thus we have
\beq
0 \le u_r \le  \max \phi \left(\frac{R_0}{|\bfx|}\right)^{n-2}
\eeq
on $\rz^n \setminus B_{R_0}$, completing the argument.

\begin{remark}\label{rmk2} We have chosen to include zero boundary
conditions at infinity in the definition of an E-inclusion in the nonperiodic
case.   This allows isolated ellipsoids to be E-inclusions and permits
a  comparison of our results with the results and conjectures
of {\sc Eshelby} \cite{Eshelby1957, Eshelby1961}.  However, it can
be seen that the theory of variational inequalities delivers the
analog of E-inclusions for bounded regions and with certain
kinds of boundary conditions.
\end{remark}

Recall that the piecewise quadratic obstacle $\phi(\bfx)$ coincides with
one of the quadratic functions $q_1(\bfx),\cdots, q_M(\bfx)$ at each
$\bfx\in \rz^n$. Let $\bbK = \{ \nabla \nabla q_i: i = 1, \dots, M
\} = \{\bfQ_1, \dots, \bfQ_N\}$, the $\bfQ_i$ being distinct
symmetric $n \times n$ matrices. Let $\Omega_i\subset \Omega_\infty$ be the largest open
set where \beq \label{eiinf} \int_{\Omega_i} |\nabla \nabla u_\infty
- \bfQ_i|^2\, d\bfx = 0. \eeq Clearly, the $\Omega_1, \dots,
\Omega_N$ are disjoint and open. It follows from $\max \phi >0$ that
$\Omega = \cup_{i=1}^N \Omega_i$ is nonempty.  We have that \beqs
\label{eq:Eininf}
\Delta u_\infty=0\;\;& \aae \oon&\;\Omega_0=\rz^n\setminus (\cup_{i=1}^N\Omega_i)\qquad\aand \\
\qquad \nabla\nabla u_\infty(\bfx)=\bfQ_i\;\;\;& \aae
\oon&\;\Omega_i \qquad \;\;\qquad \forall\,i=1,\cdots, N.\nonumber
\eeqs

We summarize below.
\begin{theorem}  \label{thm5}
Consider the variational inequality problem~\eqref{eq:fbv1} with a
piecewise quadratic obstacle  $\phi$ and define the limiting
minimizer $u_\infty$ and open disjoint sets $\Omega_i$ ($i=1,\cdots,
N$) as above. Then, the  set $\Omega=\cup_{i=1}^N\Omega_i$ is an
E-inclusion.
\end{theorem}

Theorem \ref{thm5} is a counterexample to the naive interpretation
of the original Eshelby conjecture
that the ellipsoid is the only shape in which a constant eigenstrain implies
constant stress in the inclusion (see Lemma \ref{lemma:31}) .
 A more careful interpretation of the Eshelby conjecture
 ({\sc Eshelby~\cite{Eshelby1961}}), including also the hypothesis of
 connectedness of the
 inclusion, can be proved in the framework of
variational inequalities. The details are presented in a separate
publication ({\sc Liu~\cite{LiuEshelbyConjecture}}), see also {\sc Kang \& Milton}~\cite{KangMiltonEC}.

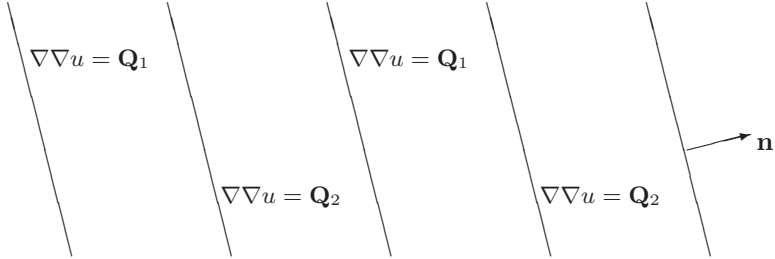
\begin{figure}[ht]
\hspace{0.15in}
\begin{center}

\setlength{\unitlength}{1.4cm}
\begin{picture}(8,3)
\put(1.0,0){\line(-1, 4){0.6}}
\put(0.6,1.8){\small \mbox{$\nabla\nabla u=\bfQ_1$}}
\put(2.5,0){\line(-1, 4){0.6}}
\put(2.4,0.5){\small \mbox{$\nabla\nabla u=\bfQ_2$}}
\put(4.0,0){\line(-1, 4){0.6}}
\put(3.6,1.8){\small \mbox{$\nabla\nabla u=\bfQ_1$}}
\put(5.5,0){\line(-1, 4){0.6}}
\put(5.4,0.5){\small \mbox{$\nabla\nabla u=\bfQ_2$}}
\put(7.0,0){\line(-1, 4){0.6}}
\put(6.78,1.0){\vector(4, 1){0.6}}
\put(7.35,1.0){ \mbox{$\bfn$}}
\end{picture}
\begin{minipage}[t]{12cm}
\caption{\it Simple laminations belong to a special family of periodic E-inclusions.}
 \label{fig:laminate}
\end{minipage}
\end{center}

\end{figure}
\begin{figure}
\begin{center}
\includegraphics[bb=0 0 50 50, viewport=0 0 570 570, scale=0.45]{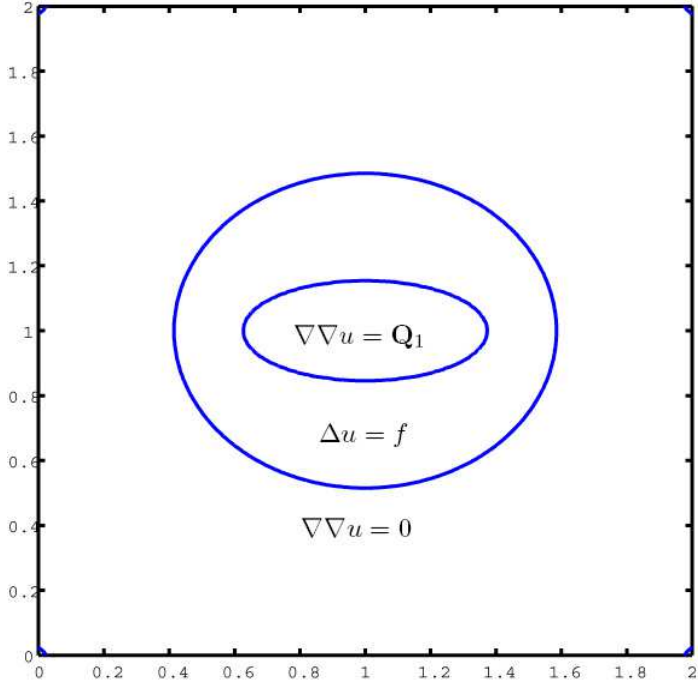}\par
\end{center}
\vspace{0.2cm}
\caption{\it Confocal ellipses are a special family of periodic E-inclusions.} \label{fig:Einconfocal2D}
 \hspace{0.2in}
\end{figure}

\renewcommand{\theequation}{\thesection-\arabic{equation}}
\setcounter{equation}{0}

\section{Examples of periodic E-inclusions} \label{sec:example}


We now consider various examples of  periodic E-inclusions.  From the
discussion above, periodic E-inclusions constructed by
Theorem~\ref{thrm:exist} can be specified by a Bravais lattice
 $\calL$, the quantity $f>0$  and a periodic piecewise
 quadratic obstacle $\phi_{per}$.
It is worthwhile noticing that from the comparison theorem (see e.g. {\sc Friedman}~\cite{Friedman1982}, page~26),
periodic E-inclusions corresponding to a fixed obstacle
satisfy $\Omega_{per}^{f_1} \subset \Omega_{per}^{f_2}$
if $f_2>f_1>0$. Also the interior of any periodic E-inclusion
cannot  contain the singular points of the obstacle on which
$\nabla \nabla \phi_{per}$ is unbounded in distributional sense.
By varying the obstacle $\phi_{per}$ and  $f$,
a large class of periodic E-inclusions can be constructed in any
dimension $n\ge 2$.
 We show a few examples  below.

The first example is a simple lamination. Let $\bfn\in \rz^n$ be a
unit vector, $f>0$,  $a<0$, and $h_{per}(x)=\max\{\half a (x+\nu)^2:\;\nu\in
\gz \}$ for $x \in \rz$.
   Consider the obstacle
\beas \phi_{per}(\bfx)=h_{per}(\bfx\cdot \bfn). \eeas
By the method given above this is a
periodic E-inclusion corresponding to $\bfQ_1=a\bfn\otimes \bfn$ and $\bfQ_2=f\bfn\otimes
\bfn$ with volume fractions $f/(f-a)$ and $-a/(f-a)$, respectively,
see Fig.~\ref{fig:laminate}. In another words a simple lamination
is a periodic E-inclusion.

We now present some numerical examples.
The  coated spheres ({\sc Hashin \& Shtrikman~\cite{HashinShtrikman1962a}}) and confocal ellipsoids
 ({\sc Milton \cite{Milton1980}}), familiar from homogenization theory,
can be constructed as a periodic E-inclusions.
The example in Fig.~\ref{fig:Einconfocal2D} is computed by using the obstacle
\beas
\phi_{per}(\bfx)=\max\{0, \half (\bfx+\bfr)\cdot \bfQ_1(\bfx+\bfr)+h_1:\;\bfr\in \calL\},
\eeas
where $h_1>0$, $\bfQ_1<0$ are appropriately chosen so that the graph
of the obstacle consists of isolated ``mountains'' emerging out of
a horizontal ``sea''. So, if $f$ is large enough,
 the minimizer $u_f$ contacts the mountains around the peaks
and the sea but is detached from the rim of singular points of $\phi_{per}$.
It can be proved that the coincident set in each unit cell is separated by two
confocal ellipsoids, by noticing   the Newtonian potential of a homogeneous solid ellipsoid
is not only quadratic inside the ellipsoid, but also  quadratic outside the ellipsoid on
the equipotential surface which is a confocal ellipsoid, see {\sc Kellogg~\cite{Kellogg1929}}.
On the other hand, if $f$ is very small or the obstacles
of \eqref{eq:phiperexample2} are considered, one obtains the Vigdergauz-type structure
as the coincident set of the variational inequality~\eqref{sec1:fbv1};
see also {\sc Grabovsky   \& Kohn~\cite{GK1995b}} for an analytic derivation of the
Vigdergauz structure. Our constructions generalize immediately to
higher dimensions.

\begin{figure}[ht]

\begin{center}
\includegraphics[bb=0 0 50 50, viewport=0 0 570 570, scale=0.45]{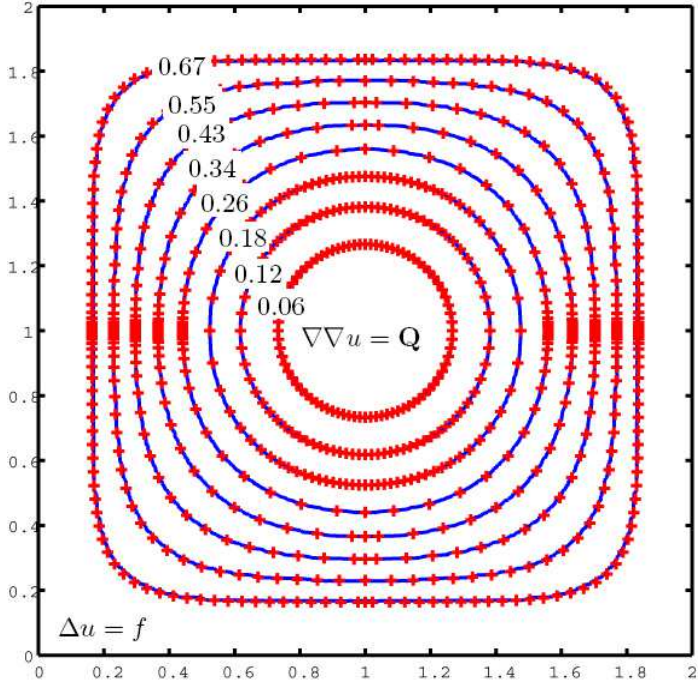}\par

\caption{\it Vigdergauz structures in a square cell
corresponding to $\bfQ=-\diag(1,1)$.
 The solid curves are our numerical results based on the
 variational inequality, and the
 ``+'' signs are the analytic solutions. The inset numbers are the volume fractions
 of the Vigdergauz structures.}
 \label{fig:Einvigtxt}

\end{center}
\end{figure}

We now describe a numerical scheme to solve the variational inequality. A detailed numerical analysis
of variational inequalities been presented in
{\sc Glowinski,  Lions \& Tremoli\`{e}res }\cite{Glowinskietal1981}.
First let us consider  the variational problem~\eqref{sec1:fbv1} with
the constraint $u\ge \phi_{per}$ neglected. Clearly the Euler-Lagrange equation of
this variational problem is the  Poisson equation
\beas
\becs
\Delta u=f&\oon\;\;Y,\\
\mbox{periodic boundary conditions}&\oon\;\;\partial Y,
\eecs
\eeas
which, according to the finite element method (see e.g. {\sc Kwon \& Bang~\cite{KwonBang2000}}), can be
discretized  as
\beqs \label{eq:Khatusmhat}
\Khat \usmhat=\fhat.
\eeqs
Here $\usmhat$, a column vector,  denotes the values of the  potential $u$ at the nodal points
in the finite element model,
 $\Khat$ and $\fhat$ are usually called the {\em stiffness matrix} and {\em loads}, respectively.
 Now let us take into account the  discretized constraint $\usmhat\ge \phihat_{per}$, where
 $\phihat_{per}$ are the values of the obstacle $\phi_{per}$ at the nodal points.
The discrete version of the variational inequality \eqref{sec1:fbv1}  becomes
the following quadratic programming problem:
\beqs \label{eq:quadprog}
\min\{\hat{G}(\usmhat)=-\half \usmhat\cdot \Khat\usmhat +\fhat\cdot \usmhat:\;\usmhat\ge \phihat_{per}\},
\eeqs
which can be easily solved using standard solvers. The following computations use
a uniform mesh with around $10^5$ nodal points. The iterations are terminated when
the relative difference between the values $\hat{G}(\usmhat)$ of two consecutive iterations
is less than $10^{-10}$. With these parameters, the iterations  converge within
a few minutes on a personal computer. The resulting periodic E-inclusion includes all
nodal points on which $|\usmhat-\phihat|$ is less than  $a\times 10^{-3}$, where $a$
is at the order of $1$.

The numerical scheme is verified by comparing the  results with the analytic
solutions for the Vigdergauz structures in two dimensions with a square unit cell
and with $\bfQ=-\diag(1,1)$.  The volume fractions were chosen to be,
from inward to outward, $0.06,\;0.12$, $0.18,\;0.26$,
$0.34, \; 0.43$, $0.55,\;0.67$.
In Fig.~\ref{fig:Einvigtxt} the solid blue curves are the numerical results while the
red ``+'' signs denote the analytic solutions. There is good agreement
between the analytical shapes and our calculated shapes.
As is well-known from the Vigdergauz construction,  E-inclusions are asymptotic to a
circle at small volume fraction and to the unit square at volume fractions approaching one.

\begin{figure}
\begin{center}
\includegraphics[bb=0 0 50 50, viewport=0 0 375 375, scale=0.70]{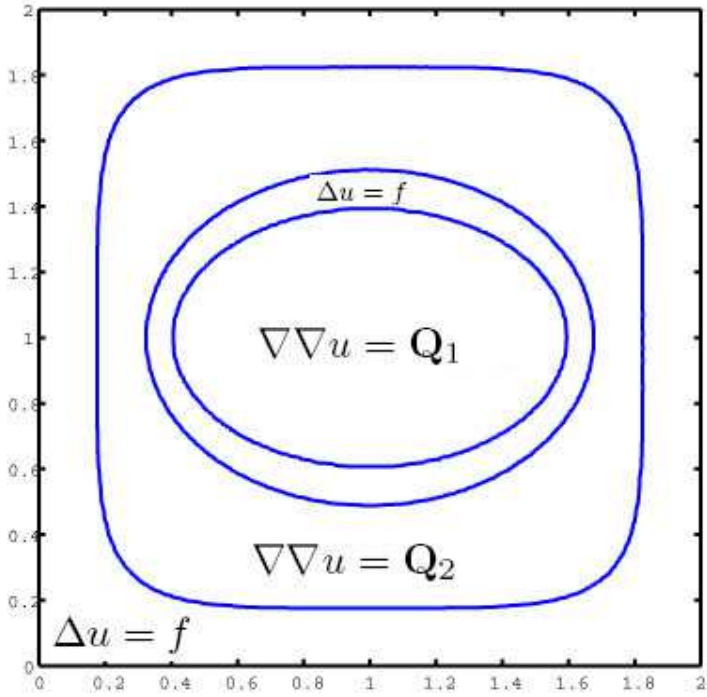}

\caption{\it A periodic E-inclusion in the case $N=2$ with two  components corresponding to
$\bfQ_1=-\diag(1,1)$ and $\bfQ_2=-\diag(2,3)$, and volume fractions $0.19$
and $0.65$, respectively.} \label{fig:Einmulticoated}
\end{center}
\end{figure}

It should be noticed that a periodic  E-inclusion may not look like an ``inclusion'' at all.
Figure~\ref{fig:Einconfocal2D} shows such an example,  the E-inclusion being
the interior of the inner ellipse and the exterior of the  outer ellipse.
A more general example is shown in Fig.~\ref{fig:Einmulticoated}.
This example is calculated using the obstacle
\beas
\phi_{per}(\bfx)=\max\{\half (\bfx+\bfr)\cdot\bfQ_1(\bfx+\bfr),\hspace{3cm}\\
\half (\bfx+\bfr)\cdot\bfQ_2(\bfx+\bfr)+h_2:\;\bfr\in 2\gz^2-\bfd \},
\eeas
where \beas
\bfQ_1=-\diag(1,1),\;\;
\bfQ_2=-\diag(2,3),\;\;
\bfd=(1,1)\;\;\aand\;\;h_2=0.2.
\eeas
The periodic E-inclusion has two components: one consists of the inner  region
 corresponding to $\bfQ_1$ and volume fraction $0.19$,
and the other is the squarish annulus  corresponding to  $\bfQ_2$ and volume fraction
 $0.65$. This type of structure can be regarded as a generalization of multi-coated
spheres ({\sc Lurie \& Cherkaev \cite{LurieCherkaev1985}}) in the periodic setting.

\begin{figure}[ht]
\begin{center}
\includegraphics[bb=0 0 50 50, viewport=0 37 400 495, scale=0.60]{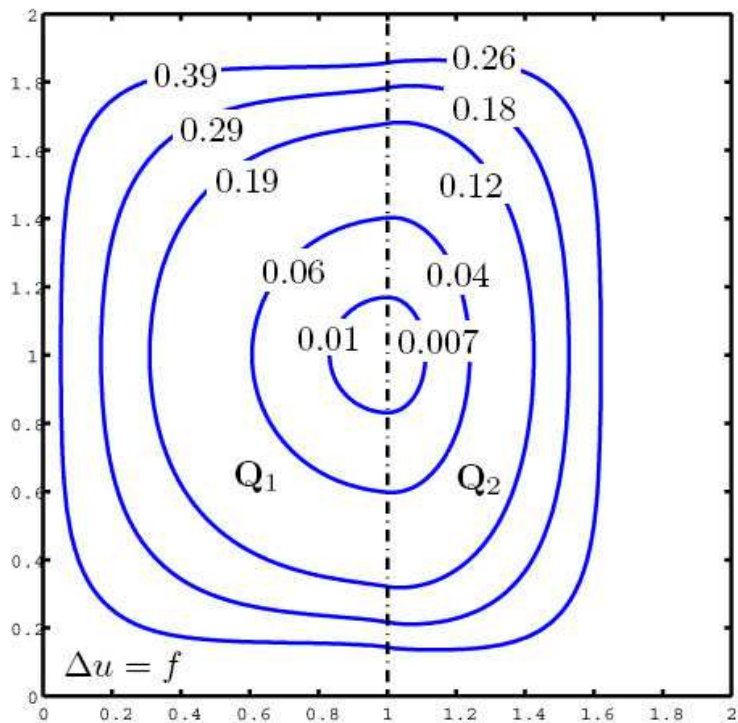}
\vspace{0.5cm}
\caption{\it A sequence of periodic E-inclusions with $N = 2$, $\bfQ_1=-\diag(1,1)$ and $\bfQ_2=-\diag(2,1)$.
In this case $\bfQ_1$ and $\bfQ_2$ differ by a rank-one
matrix and the two parts of E-inclusion are separated by a plane.
The volume fractions from inward to outward are $(0.01, 0.007)$,
$(0.06, 0.04)$, $(0.19, 0.12)$, $(0.29, 0.18)$,
$(0.39, 0.26)$. }
 \label{fig:Einlam}

\end{center}
\end{figure}

An interesting scenario is plotted in
Fig.~\ref{fig:Einlam}. Periodic E-inclusions in this figure,
 corresponding to two different matrices $\bfQ_1$ and $\bfQ_2$,
have nevertheless only one connected component in one unit cell.
Periodic E-inclusions of this kind can be
constructed by using  the obstacle
\beas 
\phi_{per}(\bfx)=
\max\{ P(\bfx+\bfr):\;\bfr\in \calL\},
\eeas
where
\beas
P(\bfx)=
\becs
\half \bfx\cdot\bfQ_1\bfx&\iif\;\bfx\cdot \bfn<0,\\
\half \bfx\cdot\bfQ_2\bfx&\iif\;\bfx\cdot \bfn\ge 0,\\
\eecs
\eeas
 $\bfQ_1,\,\bfQ_2< 0$ and
$\bfQ_1-\bfQ_2=b\bfn\otimes \bfn$ for some $b\in \rz$ and unit vector $\bfn\in \rz^n$.
Inside such a periodic E-inclusion, there is a plane interface
with normal $\bfn$ that
separates  $\nabla\nabla u=\bfQ_1$ and  $\nabla\nabla u=\bfQ_2$.
 Figure~\ref{fig:Einlam} is plotted by using
\beas
\bfQ_1=-\diag(1,1)
\qquad \aand\qquad
 \bfQ_2=-\diag(2,1).
\eeas
 The periodic E-inclusions
corresponding to $(\bfQ_1, \bfQ_2)$, from inward to outward,
 have volume fractions
$(0.01, 0.007)$, $(0.06, 0.04)$, $(0.19, 0.12)$, $(0.29, 0.18)$,
$(0.39, 0.26)$.

\begin{figure}

\begin{center}
\includegraphics[bb=0 0 50 50, viewport=0 0 570 570, scale=0.5]{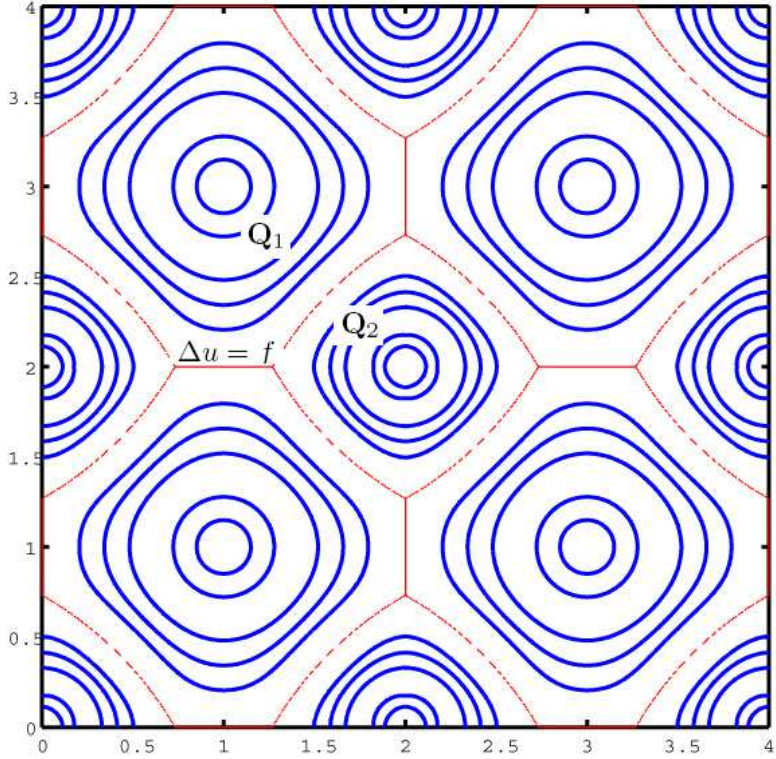}\par

\caption{\it A sequence of periodic E-inclusions with $N=2$,
$\bfQ_1=-\diag(1,1)$ and $ \bfQ_2=-\diag(2,2)$.
The volume fractions from inward to outward are
$(0.02,\; 0.01)$, $(0.07,\;0.03)$, $(0.22,\;0.09)$, $(0.33,\;0.14)$,
$(0.45,\;0.19)$.  The figure shows four unit cells.} \label{fig:Eininfty}

 \end{center}
\end{figure}

 We can construct  periodic E-inclusions with multiple components of
 a very different topology from Fig.~\ref{fig:Einmulticoated}.
Consider the obstacle
\beqs \label{phiY2}
\phi_{per}(\bfx)=
\max\{\half (\bfx-\bfd_i+\bfr)\cdot\bfQ_i(\bfx-\bfd_i+\bfr) :\qquad \nonumber \\
\;i=1,\cdots, N;\;\bfr\in \calL\},
\eeqs where $\bfd_1,\cdots, \bfd_N\in \rz^n$.
Figure~\ref{fig:Eininfty} shows examples of this kind,
corresponding to
 $\calL=2 \gz^2$, $N=2$, and
 \beas
\bfQ_1=-\diag(1,1),
\;\;
\bfQ_2=-\diag(2,2),
\;\;
  \bfd_1=[1,\;1],
\;\;   \bfd_2=[2,\;2].
\eeas
Note that four unit cells are plotted in the figure. Each
periodic E-inclusion has two components in one unit cell  corresponding to
$\bfQ_1$ and $\bfQ_2$, respectively.
The volume fractions, from inward to outward,
are $(0.02,\; 0.01)$, $(0.07,\;0.03)$, $(0.22,\;0.09)$, $(0.33,\;0.14)$,
$(0.45,\;0.19)$. The red curves delimit the singular
points of the obstacle which can never intersect the interior of
an E-inclusion.
Thus, the boundaries of the E-inclusions approach the red curves
since the total
volume fractions of the E-inclusions approach  $1$
as  $f\to +\infty$, as is implied by equation~\eqref{eq:fvf}.

 \begin{figure}[ht]
 \begin{center}
\includegraphics[bb=0 0 50 50, viewport=20 1 570 570,  scale=0.5]{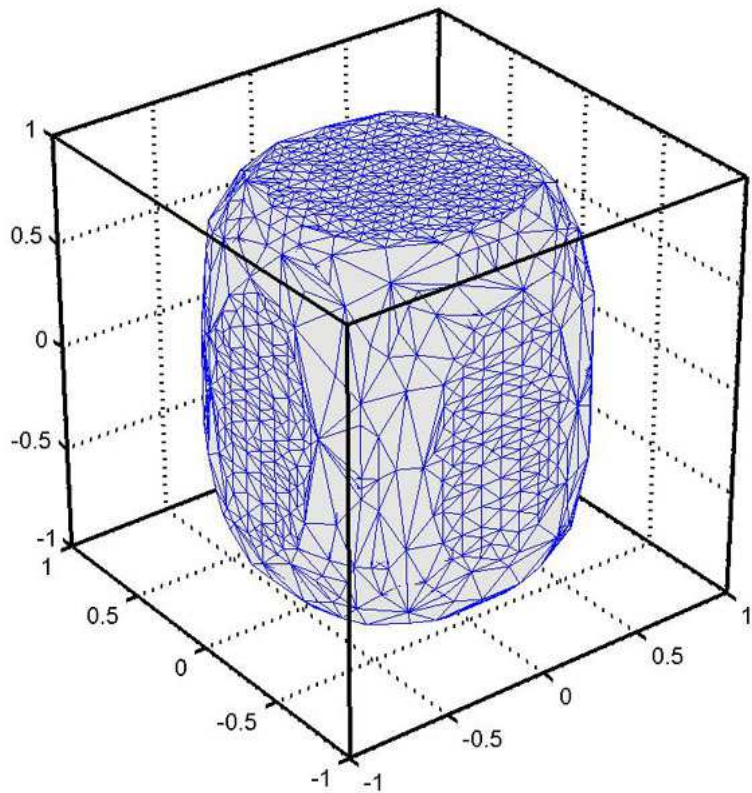}\par

\caption{\it A periodic E-inclusion corresponding to  $N = 1$,  $\bfQ_1 = \bfQ=-\diag(3,3,1)$,
and volume fraction $0.37$. The truncated parts are nearly flat.} \label{fig:Ein3DQ1}

\end{center}
\end{figure}

\begin{figure}
 \begin{center}
\includegraphics[bb=0 0 50 50, viewport=1 1 570 570, scale=0.55]{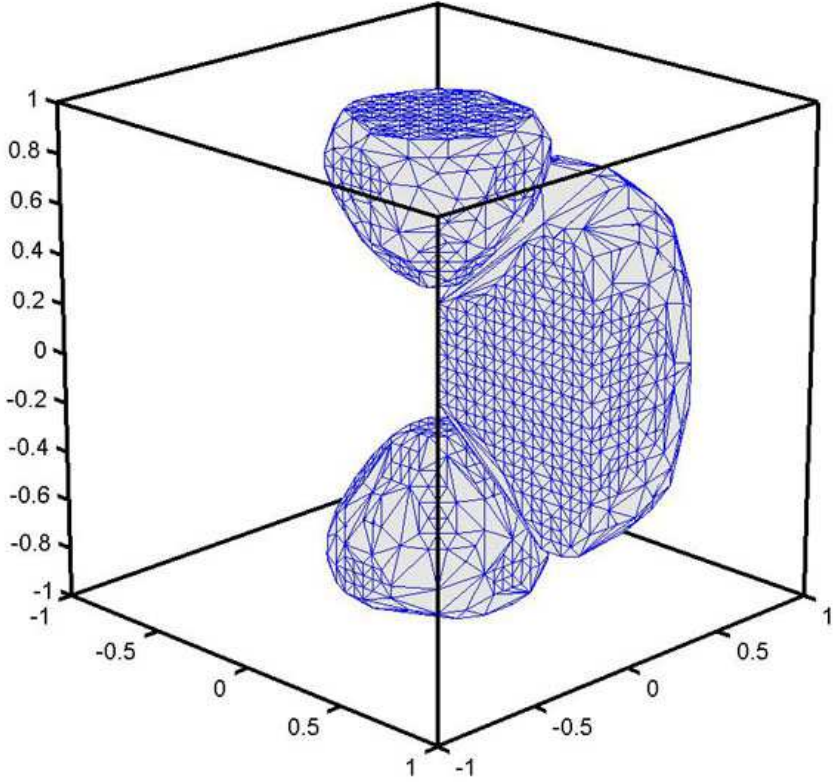}\par
\begin{minipage}[t]{12.0cm}
\caption{\it A periodic E-inclusion with $N = 3$ having three components in the unit cell
and $\bfQ_1=-\diag(1,1,1)$, $\bfQ_2=\bfQ_3=-\diag(3,3,1)$.
The top and bottom components corresponding to $\bfQ_2$ and $\bfQ_3$
 are mirror symmetric and have
the same volume fraction $0.03$, and the middle component has volume fraction
$0.35$. Only one fourth of the middle component
is plotted in the figure. The full middle component is shown separately
in Fig.~\ref{fig:Ein3DQ1Q2Q3Q1} } \label{fig:EinQ1Q2Q3}
\end{minipage}
\end{center}
\end{figure}

 \begin{figure}
 \begin{center}
\includegraphics[bb=0 0 50 50, viewport=20 1 570 570,  scale=0.55]{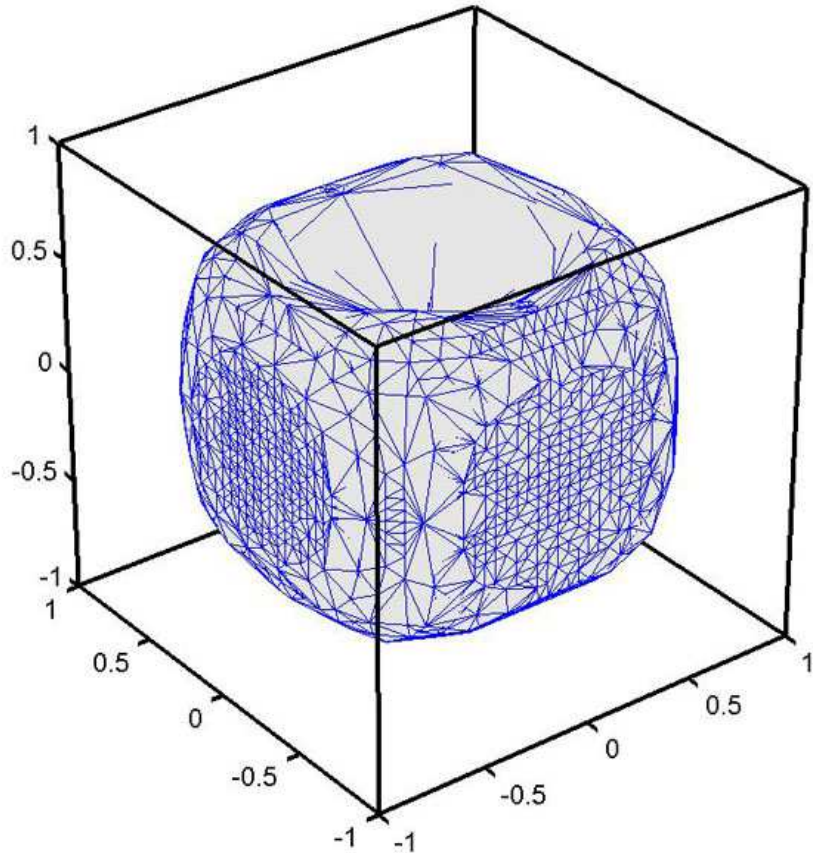}\par

\caption{\it The middle component of the periodic E-inclusion in Fig.~\ref{fig:EinQ1Q2Q3}
shown separately.  There is a depression at the top and bottom but the
shape is simply connected.} \label{fig:Ein3DQ1Q2Q3Q1}

\end{center}
\end{figure}

\begin{figure}
\begin{center}
\includegraphics[bb=0 0 50 50, viewport=1 1 570 600, scale=0.50]{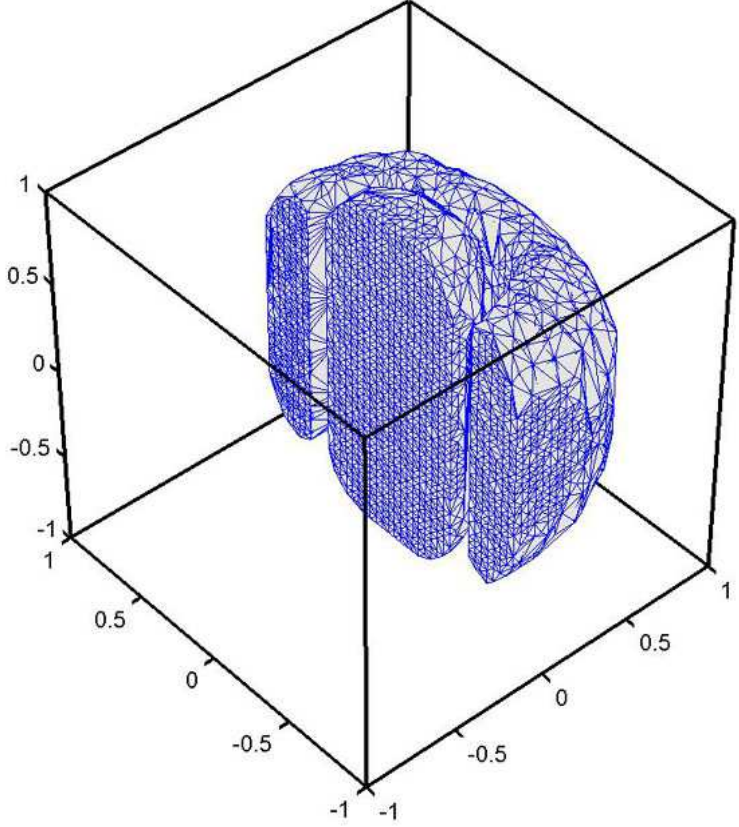}\par

\caption{\it A periodic E-inclusion for the case $N=2$
with one component surrounding the other in the unit cell.
Only  half of the E-inclusion is shown.
The inner and outer components correspond to matrices $\bfQ_2=-\diag(3,3,1)$, $\bfQ_1=-\diag(1,1,1)$
and have volume fractions $(0.11, 0.40)$, respectively.
 See Fig.~\ref{fig:Ein3Dconfocal2}
for top view.} \label{fig:Ein3Dconfocal}

\end{center}
\end{figure}

\begin{figure}
\begin{center}
\includegraphics[bb=0 0 50 50, viewport=1 0 500 500, scale=0.50]{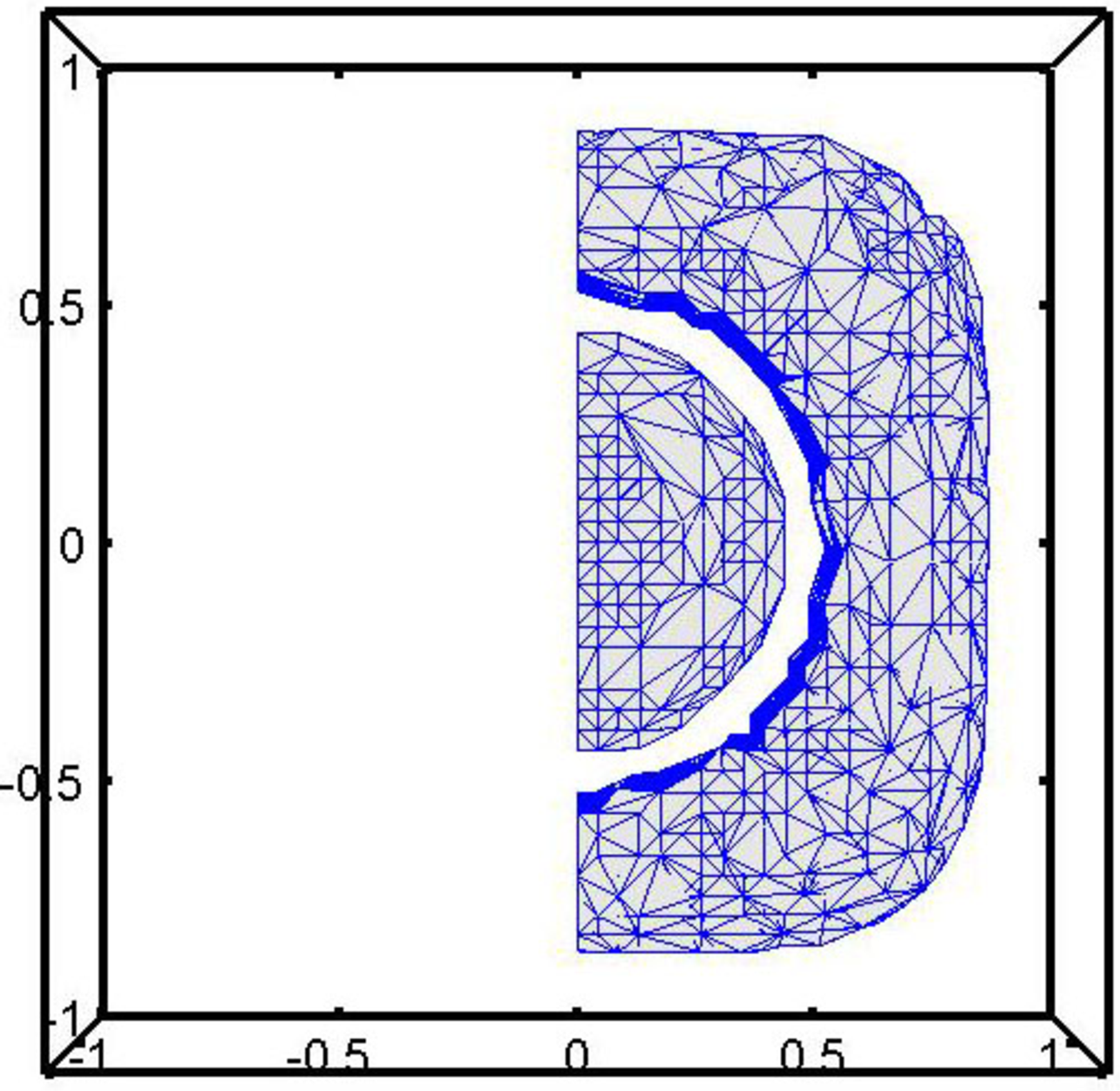}\par
\end{center}
\caption{\it Top view of the periodic
 E-inclusion in Fig.~\ref{fig:Ein3Dconfocal}. As in Fig.
 \ref{fig:Ein3Dconfocal} only half of the inclusion is shown. } \label{fig:Ein3Dconfocal2}
\end{figure}

The numerical scheme~\eqref{eq:Khatusmhat} can also be carried out in three dimensions.
The meshes used in three dimensions are not as dense as those in two dimensions. So,
the computed E-inclusions are less smooth than those in two dimensions.
 Simplified computations  have been performed
 and approximate periodic 3-D E-inclusions with cubic symmetry have been given
  in {\sc Liu, James \& Leo~\cite{LiuJamesLeo2007}}. There it was noted that some
   periodic E-inclusions
 are well approximated by {\it generalized ellipsoids} defined by
  \beas
GE(\alpha)=\{(x_1,x_2, x_3):\;\frac{x_1^{\alpha}}{a_1^{\alpha}}+\frac{x_2^{\alpha}}{a_2^{\alpha}}+\frac{x_3^{\alpha}}{a_3^{\alpha}}\le 1\}.
  \eeas
We then optimized the index ${\alpha}$ such that $GE({\alpha})$ is the best approximation according
to certain criterion. This formula can interpolate an ellipsoid and a cube.
In present approach no assumption are made about the shape of the E-inclusions
to be calculated.

Not surprisingly,  the different scenarios represented  in figures~\ref{fig:laminate}-\ref{fig:Eininfty}
are all realizable in three dimensions.
 Three typical examples are selected here.
In Fig.~\ref{fig:Ein3DQ1} a periodic E-inclusion
is calculated in the cubic unit cell $(-1,\,1)^3$ with  the obstacle~\eqref{eq:phiperexample2}
 and $\bfQ=-\diag(3,3,1)$.  The volume fraction of this E-inclusion is $0.37$. The tendency of
the boundaries of the E-inclusion to become flatter when they come closer to each other
is more obvious in three dimensions. The mesh in this and following figures does not represent
the actual mesh used in the computation but is merely used for visualization.

A three-component periodic E-inclusion is plotted in Fig.~\ref{fig:EinQ1Q2Q3}. It
is  calculated using the obstacle
\beas \label{phiY2}
\phi_{per}(\bfx)=\max\{\half (\bfx+\bfr+\bfd_i)\cdot\bfQ_i(\bfx+\bfr+\bfd_i)
:\;i=1,2,3,\;\bfr\in 2\gz^3 \},
\eeas
 where $\bfQ_1=-\diag(1,1,1)$, $\bfQ_2=\bfQ_3=-\diag(3,3,1)$, $\bfd_1=(0,0,0)$,
$\bfd_2=(0,0, 0.5)$ and $\bfd_3=(0,0,-0.5)$. The top and bottom components
corresponding to $(\bfQ_2,\bfQ_3)$ have
the same volume fraction $0.03$ and the middle component corresponding to $\bfQ_1$  has volume fraction $0.35$.
Note that only one fourth of the middle component is plotted in Fig.~\ref{fig:EinQ1Q2Q3}.
The middle component is plotted separately in Fig.~\ref{fig:Ein3DQ1Q2Q3Q1}.
A final example is shown in Fig.~\ref{fig:Ein3Dconfocal}, which is calculated with
the obstacle \beas
\phi_{per}(\bfx)=\max\{\half (\bfx+\bfr)\cdot\bfQ_1(\bfx+\bfr),\;
\half (\bfx+\bfr)\cdot\bfQ_2(\bfx+\bfr)+h_2:\;\bfr\in 2\gz^3 \},
\eeas
where  $\bfQ_1=-\diag(1,1,1)$, $h_2=0.2$ and $\bfQ_2=-\diag(3,3,1)$.
Only  half of the E-inclusion is plotted in this figure.
The inner and outer components in the figure correspond to $(\bfQ_2, \bfQ_1)$ and have volume fractions
$(0.11,0.40)$, respectively. The top view is shown in Fig.~\ref{fig:Ein3Dconfocal2}.

\renewcommand{\theequation}{\thesection-\arabic{equation}}
\setcounter{equation}{0}
\section{Applications } \label{sec:appl}



In this section, we use periodic E-inclusions to
solve problems for two-phase composites.
From Definition~\ref{def:Ein} a periodic
E-inclusion is associated to matrices $\bbK$ and volume fractions $\Theta$.
In general  periodic E-inclusions having $N' \le N$ distinct
matrices in $\bbK$ can be  used to solve problems for $(N'+1)$-phase composites.
For two-phase composites we need only periodic E-inclusions with
$\bbK = \{\bfQ, \bfQ, \dots, \bfQ \}$.
Applications of periodic E-inclusions to multi-phase composites
 are presented in a separate publication ({\sc Liu~\cite{LiuMultiphase}}).

\subsection{Periodic Eshelby inclusion problems and  effective
properties of two-phase composites} \label{sec:hominhom}

Our first observation is that some effective properties of
composites having one phase made with periodic E-inclusions can be
easily calculated. Let $\bbL$ be the collection of all symmetric
tensors $\bfL:\rz^{n\times n}\to \rz^{n\times n}$ which are either
positive definite or elasticity tensors. We consider a periodic
two-phase composite defined by \beqs \label{eq:bfLXI}
\bfL(\bfx,\Omega)=\left\{
\begin{array}{ll}
  \bfL_1\in \bbL  & \qquad \bfx\in \Omega, \\
  \bfL_0\in \bbL & \qquad \bfx\in Y\setminus \Omega, \\
 \end{array}\right.
  \eeqs
where the notation is as above, $\Omega \subset Y$ is measurable,
and $Y \subset \rz^n$ is an open unit cell.


 Consider the minimization problem
\beqs \label{Castdef:min}
J_\ast( \bfL,\, \bfF,\, \Omega)=
\min \bigg\{
\half\inttbar_Y (\nabla \bfv +\bfF)\cdot
\bfL(\bfx,\Omega)(\nabla \bfv+\bfF)\dx:\qquad \\
\bfv\in W^{1,2}_{per}(Y;\rz^n)
\bigg\}. \nonumber
\eeqs
Physically,  in the case of linearized elasticity
 $J_\ast( \bfL,\, \bfF,\, \Omega)$ is the elastic
energy density induced by an applied average strain $\bfF\in \rz^{n\times n}$.
The effective  tensor  $\bfL^e(\Omega)$ is defined as ({\sc Christensen~\cite{Christensen1979}})
  \beqs \label{Castdef}
\half  \bfF\cdot \bfL^e(\Omega)\bfF=J_\ast( \bfL,\, \bfF,\, \Omega)
 \qquad \forall\,\bfF\in \rz^{n\times n}.
 \eeqs
From standard arguments in the calculus  of variations  ({\sc Evans~\cite{Evans1998}}),  a minimizer
 of problem~\eqref{Castdef:min} exists and necessarily satisfies
the Euler-Lagrange equation
\beqs \label{ELuI}
\becs
\diverg\big[\bfL(\bfx, \Omega) ( \nabla
\bfv+\bfF)\big]=0\qquad
& \oon \;Y,\\
\mbox{periodic boundary conditions} &\oon \;\partial Y.\\
\eecs
\eeqs
We are interested in calculating the effective  tensor $\bfL^e(\Omega)$.
Problem~\eqref{ELuI} is referred to as
the {\em inhomogeneous} Eshelby inclusion  problem
in a periodic setting (cf., equation \eqref{problem:inhomrz}).

The relation between periodic E-inclusions and problem~\eqref{ELuI}
can be uncovered  by adapting a well-known  argument of {\sc Eshelby~\cite{Eshelby1957}}.
   We begin with the {\em homogeneous} Eshelby inclusion problem,
 \beqs \label{homellipsys}
\becs \diverg\big[\bfL_0 \nabla \bfv+\bfP\chi_{\Omega}\big]=0\qquad
& \oon \;Y,\\
\mbox{periodic boundary conditions} &\oon \;\partial Y,\\
\eecs
\eeqs
where  $\bfP\in \rz^{n\times n}$ is given and  $\bfv \in W^{1,2}_{per}(Y, \rz^n) $ is the unknown.
Equation (\ref{homellipsys}) is understood in the sense of distributions.
 Below, we
sometimes write $\bfv(\bfx, \bfP)$ to emphasize the (linear) dependence of
$\bfv$ on $\bfP$.
Further, motivated by the convenient property of ellipsoids employed by Eshelby,
we assume that $\Omega$ and $\bfP$ are such that there is a solution $\bfv$ of
the homogeneous problem~\eqref{homellipsys}
satisfying
\beqs \label{problem:over5}
\nabla \bfv(\bfx, \bfP)=-(1-\theta)\bfR \bfP\; &\oon \;\Omega,
\eeqs
where $\theta=|\Omega|/|Y|$ is the volume fraction of the inclusion,
and the linear mapping
\beqs \label{eq:bfRdef}
\bfR \bfP = \frac{-1}{1-\theta}\,\inttbar_\Omega \nabla \bfv(\bfx, \bfP)\dx
\eeqs
is  symmetric and depends on $\Omega$ (for the symmetry, see \eqref{eq:RFourier}).
 From equations~\eqref{homellipsys} and \eqref{eq:bfRdef}, it follows that
 \beqs \label{eq:elenergyP}
 \inttbar_Y \nabla \bfv(\bfx,\bfP) \cdot\bfL_0 \nabla \bfv(\bfx,\bfP) d\bfx
 =\theta(1-\theta)\bfP\cdot \bfR\bfP \ge 0\;\; \forall\;\bfP\in\rz^{n\times n}.\qquad
 \eeqs
Together with equation~\eqref{problem:over5}
 and following the  Eshelby's argument,
 we now observe that a solution of problem~\eqref{homellipsys}
 also solves problem~\eqref{ELuI} under restrictions given below.
 To see this, let us formally rewrite equations~\eqref{homellipsys} and~\eqref{ELuI} in a less concise form as
 \beqs \label{jumpform1}
\left\{ \begin{array}{ll}
  \diverg[\bfL_0\nabla \bfv]=0 & \iin\; Y\setminus \Omega, \\
 \diverg[\bfL_0\nabla \bfv]=0 & \iin\;  \Omega, \\
  \jumpl \bfL_0\nabla \bfv+\bfP\chi_\Omega \jumpr  \bfn=0\qquad & \oon\; \partial \Omega, \\
 \end{array}\right.
\eeqs and
\beqs \label{jumpform2}
\left\{ \begin{array}{ll}
  \diverg[\bfL_0\nabla \bfv]=0 & \iin\;Y\setminus \Omega, \\
 \diverg[\bfL_1\nabla \bfv]=0 & \iin\;  \Omega, \\
  \jumpl \bfL(\bfx,\Omega)(\nabla\bfv+\bfF) \jumpr \bfn=0 \qquad & \oon\; \partial \Omega, \\
 \end{array} \right.
 \eeqs
 respectively,
 where $\jumpl \cdot \jumpr$ denotes the jump across the $\partial \Omega$.
By matching the jump conditions in~\eqref{jumpform1}
and~\eqref{jumpform2}, direct calculations show that if $\bfv$
satisfies all equations in~\eqref{jumpform1} and
equation~\eqref{problem:over5}, then $\bfv$ also satisfies all
equations in~\eqref{jumpform2} for  $\bfF$ satisfying \beqs
\label{FPeq} \dbfL\bfF=(1-\theta)\dbfL\bfR\bfP-\bfP, \eeqs where
$\dbfL=\bfL_0-\bfL_1$.  Properly interpreted, this formal argument
can be made rigorous. More specifically, the weak form of
\eqref{homellipsys} is \beqs \label{eq:weakhomo} \inttbar_Y
(\bfL_0\nabla \bfv+\bfP\chi_\Omega)\cdot\nabla \bfw \dx=0\qquad
\forall\;\bfw\in W^{1,2}_{per}(Y;\rz^n). \eeqs By equations
\eqref{problem:over5} and \eqref{FPeq}, equation \eqref{eq:weakhomo}
can be rewritten as \beas \label{eq:weakinhomo} \inttbar_Y
[\bfL_0\nabla \bfv-\dbfL(\nabla \bfv+\bfF)\chi_\Omega]\cdot\nabla
\bfw \dx=0\qquad \forall\;\bfw\in W^{1,2}_{per}(Y;\rz^n), \eeas
which is exactly the weak form of \eqref{ELuI}.
Also, the energy of the inhomogeneous problem~\eqref{ELuI}  can be conveniently written as
 \beqs \label{energy:inhom}
2J_\ast(\bfL,\bfF,\Omega)&= &
 \inttbar_Y  (\nabla \bfv+\bfF)\cdot \bfL(\bfx,\Omega)(\nabla \bfv+\bfF)\dx \nonumber\\
&=&\inttbar_Y  \bfF\cdot \bfL(\bfx,\Omega)(\nabla \bfv+\bfF)\dx \nonumber \\
&=&\inttbar_Y  \bfF\cdot \bfL_0\bfF\dx-
\inttbar_Y  \bfF\cdot \dbfL(\nabla \bfv+\bfF)\chi_\Omega\dx \nonumber\\
&=& \bfF\cdot \bfL_0\bfF+\theta \bfP\cdot \bfF,
    \eeqs
 where
$\bfF$ and $\bfP$ are related by equation~\eqref{FPeq}.


As emphasized above, equation~\eqref{problem:over5}
is not true unless the inclusion $\Omega$
is very special. We now show that periodic E-inclusions given by Theorem~\ref{thrm:QEin}
are indeed  such special inclusions in many interesting situations. First we explain
the relation between the scalar and vector-valued problems.

\begin{lemma} \label{lemma:31}
Let
$u\in W^{2,2}_{per}(Y)$ be a distributional solution of problem
\beqs \label{problem:poisson}
\becs
\Delta u=\theta-\chi_{\Omega} &\oon\;Y,\\
\mbox{\rm periodic boundary conditions} & \oon\;\partial Y.\\
\eecs \eeqs
Denote by $\delta_{ij}$ ($i,j=1,\cdots, n$) the components of the identity
matrix $\bfI\in \rz^{n\times n}$.
 If\; $\bfL_0\in \bbL$, and
\beqs \label{eq:bfL2} (\bfL_0)_{piqj}=\mu_1
\delta_{ij}\delta_{pq}+\mu_2\delta_{pj}\delta_{iq}+\lambda
\delta_{ip}\delta_{jq}, \eeqs then
 \beqs \label{eq:bfuP1}
\bfv(\bfx, \bfP)=\frac{\bfP\nabla u(\bfx)}{\lambda+\mu_1+\mu_2}
 \eeqs
solves problem~\eqref{homellipsys} for $\bfP=\bfI$.  If $\mu_2 + \lambda = 0$
 then $\bfv$ defined by \eqref{eq:bfuP1} solves
 problem~\eqref{homellipsys} for every $\bfP \in \rz^{n \times n}$.
\end{lemma}

\begin{proof} Note that  $\bfL_0\in \bbL$ implies the constants $\mu_1$,
 $\mu_2$ and $\lambda$ necessarily satisfy
   $\mu_1\ge \mu_2$, $\mu_1+\mu_2>0$ and
$\lambda>-\frac{\mu_1+\mu_2}{n}$. Since $(\bfL_0)_{piqj}=\mu_1
\delta_{ij}\delta_{pq}+\mu_2\delta_{pj}\delta_{iq}+\lambda
\delta_{ip}\delta_{jq}$,  equation~\eqref{homellipsys} can be
formally written as \beqs \label{eq:elP0} \;\mu_1
(\bfv)_{p,ii}+(\mu_2+\lambda) (\bfv)_{q,qp}
+(\bfP)_{pi}(\chi_\Omega)_{,i}=0. \eeqs It is easy to verify by
direct calculation that $\bfv$ defined in equation~\eqref{eq:bfuP1}
satisfies equation~\eqref{eq:elP0} if $\bfP=\bfI$. If $\mu_2 +
\lambda = 0$ then $\bfv$ satisfying \eqref{eq:bfuP1} also satisfies
equation~\eqref{eq:elP0} for all $\bfP \in \rz^{n\times m}$. This
formal calculation can be made rigorous since solutions of
equation~\eqref{problem:poisson} are in $W^{2,2}_{per}(Y)$.
\end{proof}

 Now we note that if $\Omega$ is
a periodic E-inclusion specified by
equation~\eqref{problem:over1norm} with $\bfQ\in \bbQ$ (cf., \eqref{eq:bbQ}),
the second equation in~\eqref{problem:over1norm}
and equation~\eqref{eq:bfuP1} imply that for any
$\bfP\in \{a\bfI:\;a\in\rz\}$,
 \beqs \label{eq:Rtheta2}
 \nabla\bfv(\bfx, \bfP)=-\frac{(1-\theta)}{\mu_1+\mu_2+\lambda}\bfP\bfQ\;\qquad\;\oon\;\;\Omega
\eeqs
and \beqs \label{eq:Rtheta3}
\bfR\bfP=\frac{\bfP\bfQ}{\mu_1+\mu_2+\lambda}.
\eeqs
If $\mu_2+\lambda=0$,  equations~\eqref{eq:Rtheta2} and \eqref{eq:Rtheta3} hold for
 all $\bfP\in \rz^{n\times n}$ by  Lemma~\ref{lemma:31}.  Therefore, under the conditions
 specified in Lemma~\ref{lemma:31}, if $\Omega$ is a periodic E-inclusion then the
homogeneous Eshelby problem can be used to solve the inhomogeneous
Eshelby problem.

In applications to elasticity it is typically of interest to solve
the inhomogeneous Eshelby inclusion problem  \eqref{ELuI} for given
elasticity tensors $\bfL_1$ and $\bfL_0$, as this is a model of an
elastic composite. The preceding result shows that there is a
periodic E-inclusion with any positive semi-definite displacement
gradient on the inclusion, for $\bfL_0$ having the form
\eqref{eq:bfL2}.  The volume fraction of this E-inclusion is
independently assignable. The form of $\bfL_0$ is sufficiently
general to include all isotropic elasticity tensors with the usual
mild restrictions.

In applications to magnetism there are two problems of greatest
interest. In ferromagnetism one usually wants to solve the
magnetostatic equation div$(-\nabla v + \bfm \chi_{\Omega}) = 0$ for
given {\it magnetization} $\bfm \in \rz^3$. This problem corresponds
to the homogeneous Eshelby inclusion problem with $m = 1,\, n = 3,\,
\bfP = \bfm,\, \bfL_0 = \bfI$.  The preceding result in the case
$\mu_2  = \lambda = 0$ shows that for any given $\bfm \in \rz^3$,
any periodic E-inclusion (of any volume fraction) has the property
that the magnetic field $-\nabla v$ is uniform on the inclusion.
Paramagnetic or diamagnetic materials are usually described by a
linear relation between magnetization and magnetic field, $\bfm =
{\cal X} (-\nabla v)$, where ${\cal X}$ is the permeability tensor,
and the governing equation is again div$((\bfI + \calX)\nabla v) =
0$ with the average field given. A two-phase composite of such
materials is described by the inhomogeneous Eshelby problem with
$\bfL_{1,2} = (\bfI+ \calX_{1,2})$. The latter also describes a
two-phase composite of conductive materials with $\bfL_{1,2}$
interpreted as the conductivity tensors and $v$ as the electric
potential.

We now return to the general case.  From Lemma~\ref{lemma:31},
 equations~\eqref{Castdef}, \eqref{energy:inhom} and~\eqref{eq:Rtheta2},
 direct  calculations reveal the following explicit form for
  the effective  tensor of a two-phase composite with one phase
  occupying a periodic E-inclusion.

\begin{theorem} \label{thrm:LeOmega}
 Consider a two-phase periodic composite described by the
 inhomogeneous Eshelby inclusion problem~\eqref{ELuI}  for
 $n=m \ge 1$, with $\bfL^e(\Omega)$ defined
 by \eqref{Castdef} and
 \beas 
\bfL(\bfx,\Omega)=
\becs
\bfL_1 \in \bbL&  \bfx\in \Omega, \\
(\bfL_0)_{piqj}=\mu_1
\delta_{ij}\delta_{pq}+\mu_2\delta_{pj}\delta_{iq}+\lambda
\delta_{ip}\delta_{jq}\in \bbL  \qquad&\bfx\in Y\setminus \Omega,
\eecs \eeas where $\Omega$ is  a periodic E-inclusion specified by
equation~\eqref{problem:over1norm} with $\bfQ \in \bbQ$ given.
\begin{enumerate}
\item If $\mu_2+\lambda=0$, then
 \beqs \label{LeIstar:full}
\bfL^e(\Omega)=
\bfL_\theta-\theta(1-\theta)\dbfL(\bfLtld_\theta+\bfY(\bfQ))^{-1}\dbfL,
\eeqs where $\theta=|\Omega|/|Y|$,
$\bfL_\theta=\theta\bfL_1+(1-\theta)\bfL_0$,
$\bfLtld_\theta=\theta\bfL_0+(1-\theta)\bfL_1$,
$\dbfL=\bfL_0-\bfL_1$, and the mapping $\bfY(\bfQ)$ in components is
\beqs  \label{sl} (\bfY)_{piqj}=-(\bfL_0)_{piqj}+\mu_1
\delta_{pq}(\bfQ^{-1})_{ij} \eeqs for invertible $\bfQ$ (see
Remark~\ref{rmk:bfY} below).
\item If $\mu_2+\lambda\neq 0$ and $\bfI\in \calR(\dbfL)=\{\mbox{the
range of the linear mapping }\dbfL\}$, then
 \beqs \label{LeIstar}
\bfF\cdot\bfL^e(\Omega)\bfF&=&\bfF\cdot \bfL_0\bfF \\
&+&
\frac{\theta(\mu_1+\mu_2+\lambda)}{(1-\theta)-(\mu_1+\mu_2+\lambda)\Tr(\dbfL^{-1}\bfI)}(\Tr \bfF)^2
\nonumber
\eeqs
for all $\bfF\in \rz^{n\times n}$ satisfying $\Tr(\bfF) \ne 0$ and
\beqs \label{eq:FQ}
\frac{\dbfL\bfF}{\Tr(\bfF)}=\frac{(1-\theta)\dbfL\bfQ-(\mu_1+\mu_2+\lambda)\bfI}
{(1-\theta)-(\mu_1+\mu_2+\lambda)\Tr(\dbfL^{-1}\bfI)}.
\eeqs
\end{enumerate}
\end{theorem}

\begin{remark} \label{rmk:bfY}
 The meaning of the term $(\bfLtld_\theta+\bfY(\bfQ))^{-1}$ in \eqref{LeIstar:full}
in the case that $\bfQ$ is not invertible is given by
\beqs \label{eq:Dinverse}
(\bfLtld_\theta+\bfY(\bfQ))^{-1}:=\lim_{\eps\searrow 0} (\bfLtld_\theta+\bfY(\bfQ+\eps \bfI))^{-1}.
\eeqs
\end{remark}

Note that because of the restriction on $\bfF$,
 equation~\eqref{LeIstar} is not sufficient to determine
all components of the  tensor $\bfL^e(\Omega)$.
But even limited explicit results on the effective tensor are
rare in the  theory of composites.

\begin{proof}
Let us first assume that $\mu_2+\lambda=0$ and that $\Omega$ is a periodic E-inclusion
corresponding to $\bfQ$ (cf., \eqref{eq:bbQ}).
{\sc Milton }(\cite{Milton2002}, page 397) has shown that
the effective tensor can be equivalently written as equation~\eqref{LeIstar:full}
in terms of  ``$\bfY$-tensor'', which satisfies
\beqs \label{eq:Ytensordef}
\bfY \left( \inttbar_\Omega \nabla \bfv\dx \right) =
&-&\inttbar_\Omega \big\{\bfL_1(\nabla \bfv+\bfF) \nonumber \\
&-&\inttbar_Y[\bfL(\bfx,\Omega)(\nabla \bfv+\bfF)]\dx \big\}\dy.
\eeqs
From equations~\eqref{FPeq}, \eqref{eq:Rtheta2} and~\eqref{eq:Ytensordef}, direct calculations show
that $\bfY$ is given by \eqref{sl} in the case that $\bfQ$ is invertible. If $\bfQ$
is not invertible, then one still recovers equation~\eqref{LeIstar:full} with
the definition given in Remark \ref{rmk:bfY}.

If $\mu_2+\lambda\neq 0$, Lemma~\ref{lemma:31} implies equation~\eqref{problem:over5}
holds for all $\bfP=a \bfI$ ($0\neq a\in \rz$) and therefore equation~\eqref{energy:inhom}
is valid for all $\bfF$ that satisfy equation~\eqref{FPeq}.
 Since $\bfI\in \calR(\dbfL)$, equations~\eqref{eq:Rtheta2} and \eqref{FPeq} imply that
\beas
(\mu_1+\mu_2+\lambda)\Tr(\bfF)=a[(1-\theta)-(\mu_1+\mu_2+\lambda)\Tr(\dbfL^{-1}\bfI)],
\eeas
and hence equation~\eqref{FPeq} can be rewritten as equation~\eqref{eq:FQ}.
Also,
\beqs \label{eq:bfFdotbfP}
\bfF\cdot \bfP&=&a\bfF\cdot\bfI=a\Tr(\bfF) \nonumber \\
&=&\frac{(\mu_1+\mu_2+\lambda) (\Tr \bfF)^2}{(1-\theta)-(\mu_1+\mu_2+\lambda)\Tr(\dbfL^{-1}\bfI)},
\eeqs
which, by equation~\eqref{energy:inhom}, implies
 equation~\eqref{LeIstar}.
\end{proof}

\begin{remark} \label{rmk:L2general}
The special form of $\bfL_0$ has played an important role in
connecting the scalar problem~\eqref{problem:poisson} and the vector
problem~\eqref{homellipsys}. The restriction on $\bfL_0$ in
Lemma~\ref{lemma:31} and Theorem~\ref{thrm:LeOmega} can be relaxed
to satisfy the weaker restriction
 \beqs \label{eq:generalL2}
(\bfL_0)_{piqj}(\khat)_i(\khat)_j(\khat)_q=\kappa(\khat)_p
\qquad\forall\,|\khat|=1\;\;\aand\;\;\bfL_0\in \bbL,
 \eeqs
 where $\kappa>0$ is a constant.
 To show this, one notices that, by Fourier expansion ({\sc Khachaturyan~\cite{Khachaturyan1983};
  Mura~\cite{Mura1987}}),
  the gradient of the solution of equation~\eqref{homellipsys} can be represented as
 \beqs \label{Fourier:Dbfu}
\left[\nabla \bfv\right]_{pi}= -\sum_{\bfk\in \calK \setminus\{0\}}
 \chihat_\Omega(\bfk) N_{pq}(\khat)(\khat)_i
(\khat)_j (\bfP)_{jq}\exp(i\bfk\cdot \bfx), \qquad
\eeqs
where $\calK$ is the reciprocal lattice of lattice $\calL$,
 $\khat=\bfk/|\bfk|$, $N_{pq}(\khat)$
is the inverse of the matrix $(\bfL_0)_{piqj}(\khat)_i(\khat)_j$,
and $\chihat_{\Omega}(\bfk)$ is the Fourier transformation of
 $\chi_{\Omega}(\bfx)$
 \beas
\chihat_\Omega(\bfk)=\inttbar_Y \chi_\Omega(\bfx)\exp(-i\bfk\cdot \bfx)\dx.
\eeas
Therefore, the linear mapping $\bfR$ of \eqref{eq:bfRdef} can always
 be represented as \beqs \label{eq:RFourier}
(\bfR)_{piqj}=\sum_{\bfk\in \calK \setminus\{0\}} \frac{1}{\theta(1-\theta)}
 N_{pq}(\khat)(\khat)_i(\khat)_j \qquad\\
 \inttbar_\Omega \inttbar_\Omega \exp(i\bfk\cdot
(\bfx-\bfx'))\dx'\dx.
\nonumber \eeqs
Similarly, the second gradient of
the solution of problem~\eqref{problem:poisson} can be represented
as \beqs\label{Fourier:DDu} [\nabla\nabla
u(\bfx)]_{pi}=-\sum_{\bfk\in \calK \setminus\{0\}}
 \chihat_\Omega(\bfk)(\khat)_i
(\khat)_p\exp(i\bfk\cdot \bfx).
\eeqs
By comparing \eqref{Fourier:Dbfu} with \eqref{Fourier:DDu},
we note that that if equation~\eqref{eq:generalL2} holds, i.e.,
 \beas \label{eq:generalL222}
 N_{pq}(\khat)\, (\khat)_q=\frac{1}{\kappa}(\khat)_p
\qquad\forall\,|\khat|=1,
 \eeas
then for $\bfP=\bfI$,
  \beqs \label{eq:bfuP2}
\nabla \bfv=\frac{\nabla\nabla u(\bfx)}{\kappa}.
 \eeqs
This shows that under the weaker
restriction \eqref{eq:bfL2} on $\bfL_0$, the scalar problem
\eqref{problem:poisson} generates a solution of homogeneous Eshelby
inclusion problem \eqref{homellipsys}.
\end{remark}

\begin{remark}
It is useful to notice that if $\bfL_0$ satisfies
\eqref{eq:generalL2}, the energy of the homogeneous Eshelby
inclusion problem~\eqref{homellipsys} for $\bfP=\bfI$ depends only
on the volume fraction of $\Omega$. To see this, we notice by
equations~\eqref{eq:bfRdef} and~\eqref{eq:bfuP2}, \beqs
\label{eq:bfRkappa0}
\bfI\cdot\bfR\bfI=\frac{-1}{(1-\theta)\kappa}\bfI\cdot\bigg[\inttbar_\Omega
\nabla \nabla u\dx\bigg] =\frac{-1}{\kappa
\theta(1-\theta)}\inttbar_Y \chi_\Omega \Delta u\dx
=\frac{1}{\kappa}, \qquad \eeqs which, together with
\eqref{eq:elenergyP}, implies \beqs \label{eq:BitCrum} \inttbar_Y
\nabla \bfv(\bfx,\bfI)\cdot\bfL_0 \nabla \bfv(\bfx,\bfI)\dx
=\frac{\theta(1-\theta)}{\kappa}. \eeqs In the context of linearized
elasticity, equation~\eqref{eq:BitCrum} is referred to as the
Bitter-Crum theorem ({\sc Bitter~\cite{Bitter1931};
Crum~\cite{Crum1940};  Cahn \& Larche~\cite{CahnLarche1984}}).
Also, from the positive semi-definiteness of
$\bfR\bfI$ (cf., \eqref{eq:RFourier}) we have
\beqs \label{eq:bfRkappa}
\bfR\bfI=\frac{1}{\kappa}\bfQ
\eeqs
for some $\bfQ \in \bbQ$ (cf., \eqref{eq:bbQ}).
\end{remark}


\subsection{Periodic E-inclusions as optimal
structures for two-phase composites} \label{sec:locminI}

In this section, we consider the minimization/maximization problems
over measurable $\Omega$ with fixed volume fraction $\theta$ (cf.,
equation~\eqref{Castdef:min}) \beqs \label{compoptvar1}
J_\theta^l(\bfF)=\inf_{|\Omega|/|Y|=\theta}
\half\bfF\cdot\bfL^e(\Omega)\bfF \eeqs and \beqs \label{compoptvar2}
J_\theta^u(\bfF)=\sup_{|\Omega|/|Y|=\theta}
\half\bfF\cdot\bfL^e(\Omega)\bfF. \eeqs If the infimum (resp.
supremum) of problem~\eqref{compoptvar1} (resp.~\eqref{compoptvar2})
is attained by  $\Omega_\star$  (resp. $\Omega^{\star}$)
 \beas
\half\bfF\cdot\bfL^e(\Omega_\star)\bfF=J_\theta^l(\bfF) \quad
 ({\rm resp.}\ \  \half\bfF\cdot\bfL^e(\Omega^\star)\bfF=  J_\theta^u(\bfF)), \eeas
then  the region $\Omega_\star$ (resp. $\Omega^{\star}$),
corresponding to the optimal composite of   least (resp. greatest)
moduli,  is   referred to as {\em an optimal microstructure}.
Below we
note that periodic E-inclusions
 are optimal  for
 $\half \bfF\cdot \bfL^e(\Omega)\bfF$ under suitable
hypotheses on the  tensors $\bfL_1$ and $\bfL_0$.


Following from Theorem~2.1. in {\sc Liu~\cite{LiuMultiphase}},
we summarize the optimality property of periodic E-inclusions
below.
\begin{theorem} \label{thrm:qEF}
Let $0<\theta<1$.
Consider a periodic composite defined for measurable $\Omega \subset Y$ by
  \beqs \label{eq:bfLXI1}
\bfL(\bfx,\Omega)=
\becs
\bfL_1 \in \bbL&  \bfx\in \Omega, \\
(\bfL_0)_{piqj}=\mu_1
\delta_{ij}\delta_{pq}+\mu_2\delta_{pj}\delta_{iq}+\lambda
\delta_{ip}\delta_{jq} \in \bbL &\bfx\in Y\setminus \Omega. \eecs
\qquad \eeqs Assume
$\calR(\dbfL)\supset \rz^{n\times n}_{sym}$.  Let $\bfF\in
\rz^{n\times n}_{sym}$, $\Tr \bfF \ne 0$, be the average applied field and \beqs
\label{eq:bfQbfF} \bfQ= \frac{\bfF}{\Tr \bfF} \left[
1 - \frac{\mu_1 + \mu_2 + \lambda}{1-\theta} \Tr (\dbfL^{-1}\bfI) \right]
+ \frac{\mu_1 + \mu_2 + \lambda}{1-\theta} \dbfL^{-1}\bfI.
\eeqs
\begin{enumerate}
  \item   For any measurable $\Omega \subset Y$ with
  $|\Omega|/|Y|=\theta$ and any $\bfF\in \rz^{n\times n}_{sym}$, \\ if  $\bfL_1\ge
  \bfL_0$,
\beqs \label{eq:optbd0lower}
\bfF\cdot\bfL^e(\Omega)\bfF-\bfF\cdot \bfL_0\bfF \hspace{4cm}\nonumber \\
\ge
\frac{\theta(\mu_1+\mu_2+\lambda)}{(1-\theta)-(\mu_1+\mu_2+\lambda)
\Tr(\dbfL^{-1}\bfI)}(\Tr \bfF)^2; \eeqs
 if   $\bfL_1\le \bfL_0$,
\beqs \label{eq:optbd0upper}
\bfF\cdot\bfL^e(\Omega)\bfF-\bfF\cdot \bfL_0\bfF \hspace{4cm}\nonumber \\
\le
\frac{\theta(\mu_1+\mu_2+\lambda)}{(1-\theta)-(\mu_1+\mu_2+\lambda)
\Tr(\dbfL^{-1}\bfI)}\Tr(\bfF)^2. \eeqs

\item Equality holds in \eqref{eq:optbd0lower} or~\eqref{eq:optbd0upper}
 for some  $\bfF\in
\rz^{n\times n}_{sym}$, $\Tr \bfF \ne 0$  and open $\Omega \subset Y$
with $|\Omega|/|Y|=\theta$
if and only if $\Omega$ is a periodic
E-inclusion corresponding to $\bfQ$ and volume fraction $\theta$.
\end{enumerate}

\end{theorem}

\begin{remark}

The bounds \eqref{eq:optbd0lower} and \eqref{eq:optbd0upper} are
special cases of the Hashin-Shtrikman bounds.
It is well-known that there exist many nonperiodic microstructures that achieve
these bounds.  For instance, coated spheres, confocal
ellipsoids, multicoated spheres, multi-rank laminations, and
Sigmund's constructions  achieve the optimal
bounds~\eqref{eq:optbd0lower} and \eqref{eq:optbd0upper} in various
cases, see {\sc Hashin \& Shtrikman~\cite{HashinShtrikman1962a,
HashinShtrikman1963}; Milton~\cite{Milton1980}; Lurie \&
Cherkaev~\cite{LurieCherkaev1985}; Allaire \&
Kohn~\cite{AK1993a,AK1993b, AK1994}; Grabovsky \& Kohn
\cite{GK1995a, GK1995b}; Sigmund~\cite{Sigmund2000}} and {\sc
Gibiansky \& Sigmund~\cite{GibianskySigmund2000}}.
\end{remark}

\begin{remark}\label{rmkltran}
If we apply linear transformations
 \beqs \label{eq:ltrans}
\bfx \longrightarrow \bfx'=\Lambda^{-1} \bfx \qquad \aand \qquad
\bfv \longrightarrow \bfv'=  \bfG^{-1}\bfv, \eeqs to equation
\eqref{ELuI}, we can generalize Theorems~\ref{thrm:LeOmega}
and~\ref{thrm:qEF} with $\bfL_0$ of form \beas (\bfL_0)_{piqj}=
\mu_1 (\bfG^T\bfG)_{pq}(\Lambda\Lambda^T)_{ij}+
\mu_2 (\bfG^T\Lambda^T)_{pj}(\bfG^T\Lambda^T)_{qi}\\
+\lambda (\bfG^T\Lambda^T)_{pi}(\bfG^T\Lambda^T)_{qj},
\eeas
where $\bfG,\Lambda\in \rz^{n\times n}$ are any nonsingular matrices. These
changes of variables give inclusions with constant field, but, in some cases,
the inclusions are not strictly E-inclusions because the
Laplacian becomes a more general elliptic operator (with constant
coefficients).   We could have
used a more general elliptic operator in the definition of E-inclusions
but the above change of variables could then be used to reduce
that case to ours.  So, we have chosen to use the simpler definition.
These changes of variables
can be used to change the periodicity and the values of the field
$\nabla \bfv$ on the inclusion.

\end{remark}

\begin{remark}
From Remark~\ref{rmk:L2general} one can generalize
Theorems~\ref{thrm:LeOmega} and \ref{thrm:qEF} to the case of
$\bfL_0$ satisfying equation~\eqref{eq:generalL2}. The linear
transformations~\eqref{eq:ltrans} can  be applied to these $\bfL_0$
to further generalize Theorems~\ref{thrm:LeOmega} and
\ref{thrm:qEF}.  We have not been able to characterize all
elasticity tensors that can be obtained by linear transformations~\eqref{eq:ltrans}
together with~\eqref{eq:generalL2}, but it seems to be a very large set.

\end{remark}

\begin{remark}
Consider  two-phase composites with conductivity tensors
 $0<\bfA_1,\bfA_2\in \rz^{n\times n}_{sym}$ and $\dbfA=\bfA_2-\bfA_1 < 0$.
We can adapt Theorems~\ref{thrm:LeOmega} and~\ref{thrm:qEF} by
setting $(\bfL_0)_{piqj}=\delta_{pq}(\bfA_2)_{ij}$ and
$(\bfL_1)_{piqj}=\delta_{pq}(\bfA_1)_{ij}$.
 After appropriate
    linear transformations (cf., Remark~\ref{rmkltran})
    and some  algebraic  calculations, from~\eqref{eq:optbd0lower} or~\eqref{eq:optbd0upper}
     we obtain
    \beas\label{bd:trace}
\Tr(\bfA_2(\bfA^e(\Omega)-\bfA_2)^{-1})&\le&
\frac{1}{\theta} \Tr(\bfA_2(\bfA_1-\bfA_2)^{-1})+\frac{1-\theta}{\theta},
\nonumber\\
\Tr(\bfA_1(\bfA_1-\bfA^e(\Omega))^{-1})&\le& \frac{1}{1-\theta}
\Tr(\bfA_1(\bfA_1-\bfA_2)^{-1})-\frac{\theta}{1-\theta}.
\eeas
These bounds, called the ``trace bounds'' in the literature, have been
previously obtained  by {\sc Milton \& Kohn~\cite{MiltonKohn1988}}
 and have been proved to be attainable by {\sc Grabovsky~\cite{Grabovsky1993}}.
\end{remark}

\renewcommand{\theequation}{\thesection-\arabic{equation}}
\setcounter{equation}{0}
\section{Summary and discussion}\label{sec:discuss}

We have shown the existence of special inclusions for which
the overdetermined problems~\eqref{sec1:over1Prz}-\eqref{eq:overnonp}
and \eqref{sec1:over1P}-\eqref{eq:overp} admit a solution.
They are constructed as the coincident set of a simple
variational inequality with respect to piecewise quadratic obstacles.
 These structures  are called {\em E-inclusions} based on
their analogy  with ellipsoids and their extremal properties
for energy minimization problems in  homogenization theory.
Important restrictions on the parameters which characterize a periodic E-inclusion, namely,
the matrices $\bbK$ and volume fractions $\Theta$,
have also been derived, see equation~\eqref{eq:QiNeccond}.
Numerical studies have revealed the diversity of periodic
and nonperiodic E-inclusions.


It is of interest in the above to know what are the restrictions
on the symmetric matrices $\bbK = (\bfQ_1, \dots, \bfQ_N)$
and volume fractions $\Theta = ( \theta_1,\cdots, \theta_N) $ for which we
can find periodic E-inclusions.
 Let  $u\in W^{2, 2}_{per}( Y)$ be  a solution of \eqref{sec1:over1P}
 associated to a periodic E-inclusion. Using $L^p$ estimates for the Laplace operator
 we see that $u$ is in fact bounded in $W^{2,p}_{per}(Y)$ for any
 $ 1\le p<\infty$ since $\Delta u$ is bounded in $L^\infty_{per}(Y)$
 ({\sc Gilbarg \& Trudinger}~\cite{GilbargTrudinger1983}, page 235).
  Then we can rescale it and get a sequence $u^{(k)}(\bfx)
 = (1/k^2) u (k\bfx)\chi_D(\bfx) $ for an open bounded domain $D$.
 The corresponding sequence of gradients $\bfv^{(k)}(\bfx) = \nabla u^{(k)} (\bfx)= (1/k) \nabla u(k \bfx)$
 is bounded in $W^{1,p}(D,\rz^n)$ for any $ 1\le p<\infty$.  The
 study of the gradient Young measure of the sequence $\bfv^{(k)}$ gives
 natural restrictions on $\bbK$ and volume fractions $\Theta$.  For this and
 other purposes it is useful to define the concept of a  sequential E-inclusion.
 A {\bf sequential E-inclusion} is a homogeneous gradient Young measure that
is generated by a sequence  bounded  in $W^{1,p}(D)$
 for any $ 1\le p<\infty$,  has zero center of mass,
  and satisfies
\beqs \label{eq:nu}
  \nu = \sum_{i=1}^N \theta_i\delta_{\bfQ_i} + \theta_0 \mu,
\eeqs
where $\theta_1,\cdots, \theta_N \ge 0$,
 $\theta_0=1-\sum_{i=1}^N\theta_i \ge 0 $,  $\theta_0p_0=-\sum_{i=1}^N \theta_i\Tr(\bfQ_i)$,
and $\mu$ is a probability measure satisfying
supp$\, \mu \subset \{X\in \rz^{n\times n}_{sym}: \Tr(X)=p_0\}$.
 These conditions are all satisfied by the gradient Young measure generated by the
 rescaled sequence $\bfv^{(k)}=\nabla u^{(k)}$.  In
 particular, the Dirac masses at $\bfQ_i$ arise from the periodic E-inclusions and the
 condition $\Tr(X)=p_0$ arises from the Poisson equation $\Delta u = p_0$
 on the complementary set. From the basic relation between gradient Young measures and
 quasiconvex functions ({\sc Kinderlehrer \& Pedregal~\cite{KinderlehrerPedregal1991, KinderlehrerPedregal1994} }),
 we have that
\beq \label{eq:psi}
 \int_{\rz^{n \times n}} \psi (X) \, d\nu(X) \ge \psi (0)
\eeq
for all quasiconvex functions $\psi: \rz^{n \times n} \to \rz$.
The notion of quasiconvexity  used here is the one appropriate
for second gradient ({\sc \v{S}ver\'{a}k~\cite{Sverak1992}}).  In particular,
we can show that, by similar arguments as in {\sc Allaire \& Kohn~\cite{AK1993a}},
 equation~\eqref{eq:psi} holds if $\psi$ is quadratic and
 is rank-one convex for symmetric rank-one
matrices.

We now show that
the previous restrictions~\eqref{eq:QiNeccond}
on $\bbK$ and $\Theta$ also follow from equation~\eqref{eq:psi}.
For any $X\in \rz^{n\times n}_{sym}$, consider the quadratic function
 $\psi(X)=\bfm\cdot (\Tr(X)X-X^2)\bfm$
for some $\bfm\in \rz^n$. Direct calculations reveal that
$\psi(X+\lambda \bfn\otimes \bfn)=\psi(X)+\lambda
(\bfm\cdot X\bfm |\bfn|^2+\Tr(X)(\bfn\cdot\bfm)^2-2(\bfn\cdot\bfm)\bfn\cdot X\bfm)$ is an affine
function of $\lambda$ for any $\bfn,\,\bfm\in \rz^n$ and therefore is
convex on symmetric rank-one matrices. (In fact, $\psi$ is a null Lagrangian
in this second gradient context.)
 An application of \eqref{eq:psi} to $\pm \psi$  shows that for a sequential
E-inclusion~\eqref{eq:nu},
\beqs \label{eq:TrXX2}
0=\int_{\rz^{n\times n}} (\Tr(X)X-X^2)d\nu(X) \hspace{5cm}\\
=\sum_{i=1}^N\theta_i(\Tr(\bfQ_i)\bfQ_i-\bfQ_i^2)
+\theta_0\int_{\rz^{n\times n}} (\Tr(X)X-X^2)d\mu(X). \nonumber
\eeqs
Since the center of mass of $\nu$ is zero, we have
\beqs \label{eq:TrXX}
\int_{\rz^{n\times n}} \Tr(X)Xd\mu(X)=p_0\int_{\rz^{n\times n}}Xd\mu(X)
=p_0[-\sum_{i=1}^N \theta_i\bfQ_i]/\theta_0.
\eeqs
The last  term in \eqref{eq:TrXX2} can be bounded using Jensen's inequality
\beqs \label{eq:X2}
\int_{\rz^{n\times n}} X^2d\mu(X)\ge \bigg[\int_{\rz^{n\times n}} Xd\mu(X)\bigg]^2
=[-\sum_{i=1}^N \theta_i\bfQ_i]^2/\theta_0^2.
\eeqs
Substituting equations~\eqref{eq:TrXX} and \eqref{eq:X2} into \eqref{eq:TrXX2}, we
obtain
\beas
\sum_{i=1}^N [\theta_0\Tr(\bfQ_i)+\sum_{j=1}^N \theta_j\Tr(\bfQ_j)]\theta_i\bfQ_i\ge
\theta_0\sum_{i=1}^N \theta_i\bfQ_i^2+[\sum_{i=1}^N \theta_i\bfQ_i]^2,
\eeas
which is identical to equation~\eqref{eq:QiNeccond}. Of course, there are many other
quasiconvex functions that could be used in equation~\eqref{eq:psi} that would
evidently give further restrictions on the $\bbK$ and $\Theta$.

As shown in recent work on optimal bounds for multiphase composites ({\sc Liu~\cite{LiuMultiphase}}),
the concept of a sequential E-inclusion is useful to characterize the
microstructures that attain the Hashin-Shtrikman bounds. The result in its strongest
form states that if $\bfL_1, \dots, \bfL_N$ satisfies $\calR(\bfL_i - \bfL_0)
\supset \rz^{n \times n}_{sym},\ i \ne 0 $ (see Theorem \ref{thrm:qEF}),
a microstructure  attains the Hashin-Shtrikman bounds if and only
if it is a sequential E-inclusion.

 E-inclusions can also be used to solve
problems on the effective behavior of nonlinear  composites.
For instance, let us consider
a periodic two-phase nonlinear composite with effective
properties defined by
\beqs \label{eq:Ie}
I^e(\bfe_0)=\min_{w\in W^{1,2}_{per}(Y) }
\inttbar_Y I(\nabla w+\bfe_0, \bfx)\dx ,
\eeqs
where $\bfe_0\in \rz^n$ is the applied average field, and
the energy function $I:\rz^n\times \rz^n\to \rz$  is given by
\beas
I(\bfe,\bfx)=\becs
I_1(\bfe)&\iif\;\bfx\in\Omega,\\
\bfe\cdot\bfI\bfe&\iif\;\bfx\in Y\setminus \Omega.\\
\eecs
\eeas
Here $I_1:\rz^n\to \rz$,  describing the
nonlinear phase on $\Omega$, is a strictly convex but not necessarily quadratic function,
and the identity matrix $\bfI$ describes the linear phase on $Y\setminus \Omega$.
If $\Omega$ is a periodic E-inclusion with matrix $\bfQ\in \bbQ$ and volume fraction
$\theta$ (cf., \eqref{problem:over1norm}), the minimization problem \eqref{eq:Ie}
 is explicitly solvable in terms of a linear combination of functions satisfying
\eqref{problem:over1norm} in spite of the nonlinearity of $I_1(\bfe_0)$.  The
details will be presented in a future publication. Essentially,
the fact that $\nabla w$ is constant on $\Omega$ reduces the nonlinear part of the
problem to an algebraic equation.   This algebraic equation
may be solvable in any nonlinear theory for which linear boundary conditions
gives a linear solution.  It is then only a matter of showing that
these boundary conditions can be satisfied.
This observation (for ellipsoids) goes back to
{\sc Hill~\cite{Hill1965}}.

\begin{figure}
\begin{center}
\includegraphics[bb=0 0 50 50, viewport=0 0 360 360, scale=0.7]{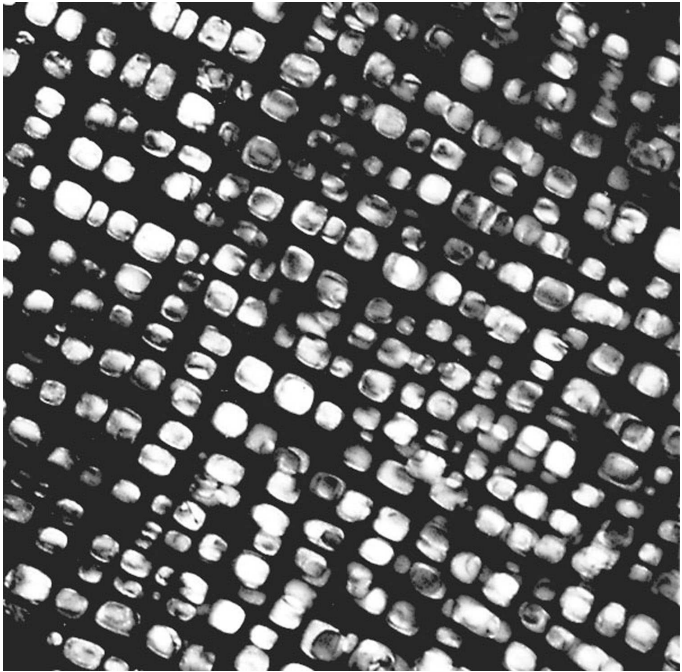}
\caption{\it Dark-field transmission electron micrographs of the Ni$_3$Ge precipitates,
see {\sc Kim \& Ardell~\cite{KimArdell2003}} for the experimental conditions. \label{fig:Ardell}}
 \end{center}
\end{figure}

Structures  that look a lot like E-inclusions are apparently seen in
nature.  For instance,  Fig.~\ref{fig:Ardell} is a dark-field electron micrograph
 of Ni$_3$Ge precipitates  in late stage coarsening of binary Ni-Ge alloys.
The experimental conditions are described in {\sc Kim \& Ardell}~\cite{KimArdell2003}.  The
transformation strain is dilatational in this case and the transformation is
cubic to cubic.
Precipitates  in this and other coarsening nickel-based superalloys form
approximately periodic arrays and have shapes like those seen in Fig.~\ref{fig:Ein3DQ1}.
Like periodic E-inclusions, these precipitates become more cuboidal at higher volume
fraction. In late stage coarsening it is accepted that the
minimization of elastic energy, both with respect to fields as well as shapes,
governs the evolution of microstructure ({\sc Jou, Leo \& Lowengrub~\cite{JouLeoLowengrub1997};
Thornton, Akaiwa \& Voorhees~\cite{ThorntonAkaiwaVoorhees2003}}).  In
addition, interfacial energy plays a role in the evolution,
but its influence is less important at the later stages of coarsening.
 E-inclusions are likely to be favored in the evolution of precipitates in these
alloys since they are optimal shapes with respect to elastic energy, as shown
in Section~\ref{sec:appl}.  To make this possible connection between E-inclusions and
Ni$_3$Ge precipitates quantitative, one should find an estimate of the effect
of the presence of interfacial energy on shape.  Also, one would also need to check
that the elasticity tensor of the matrix phase exceeds that of the precipitate
(or generalize our results) and generalize our results to allow a
cubic matrix phase.

Finally we remark that the analogy between ellipsoids
 and periodic E-inclusions  is not perfect in the sense
 that $\nabla \bfv$ is uniform on the ellipsoid  in  problem~\eqref{problem:magel} for {\em any}
matrix $\bfP$, whereas periodic E-inclusions have this property only
for the  matrix $\bfP=\bfI$, unless the tensor $\bfL_0$ has a
special form. Nevertheless,  we anticipate periodic E-inclusions can
find wide applications in the theories of micromechanics, composites
and fracture mechanics as does the ubiquitous  {\em Eshelby's
solution} in  these fields.

\vspace{5mm}
\noindent {\bf Acknowledgment}.  We thank the reviewers for valuable
comments.
This work was supported by
AFOSR (Game Changer, GRT00008581/ RF60012388, and Computational Mathematics,
FA9550-09-0339), the MURI program
(ONR N000140610530, ARO W911NF-07-1-0410),  DOE (DE-FG02-05ER25706), NSF (DMS-0757355) and NSF (CMMI-1101030).
L. L. also acknowledges the startup support from University of Houston and the support of the Texas Center for Super-conductivity, University of Houston (TcSUH).




\address{ Liping Liu\\
Department of Mechanical Engineering \\
University of Houston\\
Houston, TX 77204\\
email:{liuliping@uh.edu}
\and
Richard D. James\\
Department of Aerospace Engineering and Mechanics\\
 University of Minnesota\\
  Minneapolis MN55455\\
email:{james@umn.edu}
\and
Perry H. Leo\\
Department of Aerospace Engineering and Mechanics\\
 University of Minnesota\\
  Minneapolis MN55455\\
email:{phleo@aem.umn.edu}}

\end{document}